\documentclass{class}
\usepackage{xcolor}
\usepackage[verbose]{hyperref}
\hypersetup{colorlinks=false,allbordercolors=blue,pdfborderstyle={/S/U/W 1}}

\usepackage{amssymb}
\usepackage{booktabs}
\usepackage{graphicx}
\usepackage{hyperref}
\usepackage{amsfonts}
\usepackage{amsmath}
\usepackage{mathtools}
\usepackage{graphicx}
\usepackage{xcolor}
\usepackage{stmaryrd} 
\SetSymbolFont{stmry}{bold}{U}{stmry}{m}{n} 
\usepackage{algpseudocode}
\usepackage{algorithm2e}
\usepackage{multirow}
\usepackage{adjustbox}

\usepackage{subcaption}

\usepackage{tikz}
\usepackage{tikz-cd}

\newcommand{\R}{\mathbb{R}}
\newcommand{\Z}{\mathbb{Z}}

\newcommand{\kk}{\kappa}

\newcommand{\cM}{\mathcal{M}}

\newcommand{\cA}{\mathcal{A}}
\newcommand{\cB}{\mathcal{B}}
\newcommand{\cC}{\mathcal{C}}

\newcommand{\sigmaV}{\sigma^{\scst V}}
\newcommand{\sigmaf}{\sigma^{f}}
\newcommand{\sigmainf}{\sigma^{\infty}}
\newcommand{\sigmaO}{\sigma^{\scst O}}

\newcommand{\dM}{d^{\scst M}}
\newcommand{\dB}{d_{\scst B}}
\newcommand{\dI}{d_{\scst I}}

\newcommand{\gammaZ}{\gamma_{\scst Z}}
\newcommand{\gammaZp}{\gamma_{{\scst Z}'}}

\newcommand{\fdp}{f^{\delta\,'}}

\newcommand{\Ydp}{Y^{\delta\,'}}
\newcommand{\piYd}{\alpha^{\scst Y\delta}}
\newcommand{\piYdp}{\alpha^{{\scst Y\delta}\,'}}

\newcommand{\dYdp}{d^{{\scst Y\delta}\,'}}
\newcommand{\fbullet}{f_\bullet}

\newcommand{\scst}{\scriptscriptstyle}
\newcommand{\Sf}{S^{\scst D(f_0)}}
\newcommand{\mf}{m^{\scst D(f_0)}}
\newcommand{\Sff}{S^{\scst D(f_0')}}
\newcommand{\mff}{m^{\scst D(f_0')}}

\newcommand{\dHM}{d^{\scst M}_{\scst H}}

\newcommand{\kkI}{\kappa_{\scst I}}

\DeclareMathOperator{\B}{B}
\DeclareMathOperator{\Ho}{H}
\DeclareMathOperator{\PH}{PH}

\DeclareMathOperator{\VR}{VR}
\DeclareMathOperator{\im}{im}

\DeclareMathOperator{\coim}{coim}
\DeclareMathOperator{\coker}{coker}
\DeclareMathOperator{\Rep}{Rep}

\newcommand{\veps}{\varepsilon}

\newcommand{\rhoV}{\rho^{\scst V}}
\newcommand{\rhoU}{\rho^{\scst U}}
\newcommand{\rhoVveps}{\rho^{\scst V}_{\veps}}

\newcommand{\Vveps}{V(\veps)}
\newcommand{\Uveps}{U(\veps)}
\newcommand{\fveps}{f_\veps}

\begin{document}

\markboth{Á. Torras-Casas,  R. Gonzalez-Diaz}
{Properties and Stability of Persistence Matching Diagrams}

%
\catchline{}{}{}{}{}
%

\title{Properties and Stability of Persistence Matching Diagrams }

\author{Álvaro Torras-Casas and Rocio Gonzalez-Diaz}

\address{
    Matemática Aplicada I,  Universidad de Sevilla, \\ 
    Avenida Reina Mercedes S/N, Sevilla, 41012,
    Spain \\
    $\{$atorras,rogodi$\}$@us.es
}

\maketitle


\begin{abstract}
We introduce persistence matching diagrams induced by set mappings of metric spaces, based on 0-persistent homology of Vietoris-Rips filtrations. 
Also, we present a geometric definition of the persistence matching diagram that is more intuitive than the algebraic one.
In addition, we show that such matching diagrams encapsulate the information of persistence morphism images, kernels and cokernels induced by embeddings.
The main result is a stability theorem for persistence matching diagrams induced by embeddings.
At the end, we adapt our stability result for set injections (not embeddings) of finite metric spaces.
\end{abstract}




\section{Introduction}

Our study starts with the case of finite metric spaces $X$ and $Z$ together with an embedding $X\subseteq Z$. 

Given a pair of finite metric spaces $X\subseteq Z$, the usual procedure to date computes the persistent homologies $\PH_*(X)$ and $\PH_*(Z)$ obtained from the respective Vietoris-Rips filtration and compares their interval decompositions (also called barcodes), $\B(\PH_*(X))$ and $\B(\PH_*(Z))$, by some distance, such as the bottleneck or Wasserstein distance (see, for example, \cite{carlsson,edelsbrunner,oudot}).
However, such comparisons might differ substantially from the underlying distribution of the metric spaces $X$ and $Z$, since both distances are based only on combinatorial comparisons between the interval decompositions of $\PH_*(X)$ and $\PH_*(Z)$.

To also consider the information of the morphism 
$f\colon\PH_*(X)\rightarrow \PH_*(Z)$ induced by the inclusion $X\subseteq Z$ in the comparison, we propose to use the block function $\cM_f$ induced by the morphism $f$, introduced in~\cite{matchings}. 
Since this paper focuses on 0-persistent homology,  the block function will be denoted as $\cM^0_f$ to highlight this fact, where $f\colon V\to U$, being $V=\PH_0(X)$ and $U=\PH_0(Z)$.
We show that $\cM^0_f$, $V$ and $U$
encode the same information as the persistence modules $\im(f)$, $\ker(f)$ and $\coker(f)$ associated with the morphism $f$.
However, the advantage of using $\cM^0_f$ is that we can use it to define a single diagram $D(f_0)$ that combines all the information from $\B(V)$, $B(U)$, $\B(\im(f))$, $\B(\ker(f))$ and $\B(\coker(f))$. 
Besides, we show that $\cM^0_f$  is a well-behaved invariant and stable under small perturbations and that 
redundant information of $X$ and $Z$ are concentrated in $D(f_0)$ around the origin. An additional advantage of $\cM^0_f$ studied in the paper is that $\cM^0_f$ can be computed in situations when there is no underlying persistence morphism, but only a set mapping $X\to Z$.

We start by introducing $\cM^0_f$, the block function between barcodes induced by a set mapping of finite metric spaces, in Section~\ref{sec:matchings}, where we include examples.
In Section~\ref{sec:geometric-interpretation}, we give a geometric interpretation of  
 $\cM^0_f$.
Later, in Section~\ref{sec:properties-stability}, we study properties of $\cM^0_f$. 
In Section~\ref{sec:0-diagrams}, we study the stability for the matching diagram
$D(f_0)$ which, by definition, is equivalent to the stability for $\cM^0_f$.
Section~\ref{sec:proofs} is devoted to some proofs of results provided in the paper,
for being rather technical.   
Section~\ref{sec:future} ends the paper with conclusions and future work. The main notations used in the paper can be consulted in Table~\ref{tab:notation}.

\begin{table}
\centering
\begin{tabular}{|c|l|}
\hline
\textbf{Notation} & \textbf{Description}
\\ \hline
$X,Z$ & finite metric spaces 
\\ \hline 
\multirow{2}{*}{$\VR(X), \VR(Z)$} & 1-skeletons of the Vietoris-Rips filtrations
\\  
&
\;\;\;  associated to $X,\, Z$
\\ \hline 
 $V_0,\, U_0$ &  $H_0(\VR(X)),\, H_0(\VR(Z))$, vector spaces  spanned by $X,\, Z$
 \\ \hline 
\multirow{2}{*}{$V, U$} & $\PH_0(X), \PH_0(Z)$, persistence modules
\\ 
&
\;\;\; induced by $\VR(X), \VR(Z)$
\\ \hline 
$\B(V),\B(U)$ & barcodes of $V,\, U$
\\ \hline 
$(S^V,m^V), (S^U,m^U)$ & 
pairs of set and multiplicity of $\B(V)$ and $\B(U)$
\\ \hline 
$X\subseteq Z$ & embedding of  metric spaces
\\ \hline 
$X\hookrightarrow Z$ & set injection
\\ \hline 
$\fbullet\colon X\to Z$ & non-expansive map (or just a set mapping)
\\ \hline 
\multirow{2}{*}{$f_0\colon V_0\to U_0$} & linear map\\
&
\;\;\; induced by a set mapping $\fbullet\colon X\to Z$
\\ \hline 
\multirow{2}{*}{$f\colon V\to U$} & persistence morphism 
\\  
&
\;\;\; induced by a non-expansive map $\fbullet\colon X\to Z$
\\ \hline 
$\cM_f^0$ & block function induced by $f_0$
\\  \hline 
$D(f_0)$ & persistence matching diagram induced by $f_0$ 
\\ \hline
\end{tabular}
\caption{Main notations used in the paper.}
\label{tab:notation}
\end{table}


\section{Block functions between barcodes induced by set mappings
}
\label{sec:matchings}

Consider a pair of finite metric spaces $X$ and $Z$ and an embedding $X\subseteq Z$. 
Thus, there is an inclusion of Vietoris-Rips filtrations $\VR(X)\subseteq \VR(Z)$.
Fixing a ground field, a morphism between the persistent homology of the Vietoris-Rips filtrations can be obtained, and the codification of this information in barcodes and block functions encapsulates rich topological information derived from the embedding $X\subseteq Z$.
Specifically, in this work, we are only interested in the 0-persistent homology, and so, without loss of generality, we only work with the 1-skeleton  of Vietoris-Rips filtrations (that we also denote as $\VR$ to unload the notation) and fix $\Z_2$ as the ground field. 
Let us denote by $\PH_0(X)$ and $\PH_0(Z)$ the 0-persistent homology over $\Z_2$ of $\VR(X)$ and $\VR(Z)$  respectively.
In this section, we explain how to compute a block function $\cM^0_f$ from the barcode $\B(V)$ to the barcode $\B(U)$, where $U=\PH_0(X)$ and $V=\PH_0(Z)$.


\subsection{Background}

First, we need to introduce some technical background from the computational topology field.
All the information can be found, for example,  in \cite{edelsbrunner, OudotBook2015} and in \cite{polterovich2020topological} where applications to metric geometry and function theory are also included.


\subsubsection{Finite metric spaces and the Gromov-Hausdorff distance}
\label{subsec:GH-dist}

Let $M$ be a finite metric space with metric $d^{\scst M}$.  
Given a pair of subsets $A,B\subseteq M$, we define their pairwise distance as the quantity 
\[
d^{\scst M}(A,B) = \min\big\{d^{\scst M}(a,b) \mbox{ for all } a \in A \mbox{ and } b \in B\big\}\,.
\] 
When considering a pair $A \subseteq M$ and a point $x \in M$, we use the notation $d^{\scst M}(x, A)$ to refer to the expression $d^{\scst M}(\{x\}, A)$.

A disadvantage of the definition of distance between subsets is that it is not effective in distinguishing potentially very different subsets.
If $A\cap B\neq \emptyset$ then $d^{\scst M}(A,B)=0$. 
For this purpose, it is more useful to consider the \emph{Hausdorff} distance
\[
d^{\scst M}_{\scst H}(A,B) = \max\big\{\max\big\{ 
d^{\scst M}(a, B) \, \vert \,
a \in A
\big\}, \max\big\{ 
d^{\scst M}(A, b) \, \vert \,
b \in B
\big\}\big\}\,,
\]
since $d^{\scst M}_{\scst H}(A,B)=0$ if and only if $A=B$.

Given a finite metric space $Z$, an isometric embedding of $Z$ into $M$ is denoted by  $\gamma_{\scst Z}\colon Z\hookrightarrow M$.
Now, consider two finite metric spaces, $X$ and $X'$. 
To compare $X$ and $X'$, we might consider  
isometric embeddings of both objects into metric spaces of $M$. 
This leads to the following definition.

\begin{definition}\label{def:gromov-hausdorff-pairs}
    The \emph{Gromov-Hausdorff} distance between $X$ and $X'$ is 
    \[
    d_{\scst GH}(X,X') = \inf_{M, \gamma_{\scst X}, \gamma_{\scst X'}}\big\{ d^{\scst M}_{\scst H}(\gamma_{\scst X}(X), \gamma_{\scst X'}(X'))\big\}\,.
    \]
Further,  consider pairs $(X,Z)$ and $(X',Z')$ where $Z$ and $Z'$ are a couple of metric spaces and where $X \subseteq Z$ and $X' \subseteq Z'$.
We define the 
Gromov-Hausdorff distance 
$d_{\scst GH}((X,Z),(X',Z'))$ between the pairs $(X,Z)$ and $(X',Z')$, as the quantity
\[
\inf_{M, \gamma_{\scst Z}, \gamma_{\scst Z'}}\big\{ \max\big\{ 
d^{\scst M}_{\scst H}(\gamma_{\scst Z}(X), \gamma_{\scst Z'}(X')),
d^{\scst M}_{\scst H}(\gamma_{\scst Z}(Z), \gamma_{\scst Z'}(Z'))
\big\}\big\}\,.
\]
\end{definition}

The Gromov-Hausdorff distance for pairs measures how far two pairs of metric spaces are from being isometric, i.e., from having the same geometric structure. 
This concept is used in Subsection~\ref{subsec:stability} to prove the stability of the 
block function $\cM_f^0$.


\subsubsection{
Persistence modules, intervals and persistent homology in dimension 0}\label{subsec:pershom}

A \emph{persistence module} $V$ indexed by $\R$ consists of a set of vector spaces $\big\{V_t\big\}_{t\in\R}$ and a set of linear maps 
$\big\{\rhoV_{st}\colon V_s \rightarrow V_t\big\}_{s\leq t}$, called the \emph{structure maps} of $V$,  satisfying that $\rhoV_{jt}\rhoV_{ij} = \rhoV_{it}$ and $\rhoV_{tt}$ being the identity map, for $i\leq j\leq t\in\R$.
Given an interval $I=[a,b)\subset \R$,
the \emph{interval module}, $\kkI$, is a particular case of persistence module consisting of $(\kkI)_t= \Z_2$ for all $t \in I$ and $(\kkI)_t= 0$ otherwise, while the linear map $\rho^{\kk_{\scst I}}_{ij}\colon (\kkI)_i\to (\kkI)_j$ is the identity map whenever $a\leq i\leq j< b$. 

In this work, for most cases, we consider only interval modules such that $a=0$.
Thus, for ease, we often denote an interval module $\kk_{[0,b)}$ by the simpler notation $\kk_{b}$ for all $b>0$.
Also, we denote by  $\kk_{\infty}$  
the persistence module that consists of 
$(\kk_{\infty})_t= \Z_2$ for all $t\geq 0$ and 
$(\kk_{\infty})_t=0$ otherwise; and $\rho^{\kk_\infty}_{ij}\colon(\kk_{\infty})_i\to (\kk_{\infty})_j$ being the identity map whenever $0\leq i\leq j$ and the zero map otherwise.

Let $Z$ be a metric space with metric $d^Z$. 
We consider the filtered graph $\VR(Z)$, the 
1-skeleton of the Vietoris-Rips complex, which is a family of graphs $\{\VR_r(Z)\}_{r\in[0,\infty)}$.
Given $r\in [0,\infty)$ the graph $\VR_r(Z)$ contains the edges $[x,y]$ such that $d^Z(x,y)\leq r$. Hence, given $r\leq r'$ we have an injection $\VR_r(Z)\hookrightarrow \VR_{r'}(Z)$.

Fixed $r\geq 0$, the \emph{0-homology group} of $\VR_r(Z)$ is a quotient group, denoted as  $\Ho_0(\VR_r(Z))$, that consists of the vector space generated by the classes associated with the connected components of $\VR_r(Z)$.
Specifically, we define $\Ho_0(\VR_r(Z))=\big\langle \pi_0(\VR_r(Z))\big\rangle$, the free $\Z_2$-vector space generated by the set $\pi_0(\VR_r(Z))$ of connected components of $\VR_r(Z)$. 

The \emph{0-persistent homology} of $\VR(Z)$, denoted as $U=\PH_0(Z)$,
is a persistence module given by the set $\big\{
U_r=\Ho_0(\VR_r(Z))
\big\}_{r\in[0,\infty)}$ 
of 0-homology groups and the set of linear maps 
$
\big\{
\rhoU_{rs}\colon U_r \to U_s
\big\}_{r\leq s}
$
that are induced by the inclusion maps
$\VR_r(Z)\xhookrightarrow{} \VR_s(Z)$
for $r\leq s$.
Intuitively, $\PH_0(Z)$ encapsulates the evolution of connected components in $\VR(Z)$. 
These components start being the isolated points of the whole dataset $Z$ and, as the filtration parameter increases, $\PH_0(Z)$ records the death values of such components. 

Given the persistence module $U=\PH_0(Z)$ with structure maps $\rho^{\scst U}$, 
the operators $\ker^\pm_{b}$ are
defined as follows, for all $b> 0$,
\[
\ker^-_b(U
)= \bigcup_{0 \leq r < b} \ker(\rho^{\scst U}_{0r})
\;\;\mbox{ and }\;\; 
\ker^+_b(U)
= \ker(\rho^{\scst U}_{0b})
\,,
\]
are subspaces of $U_0$ and encapsulate the classes in $\PH_0(Z)$ that die at $r$ for $r<b$ and  $r\leq b$, respectively.


\subsubsection{Multisets, p.f.d. persistence modules and barcodes}

It is known that, in general, a persistence module $V$ has a unique decomposition as a direct sum of interval modules (see~\cite{crawley-boevey-2015}). 
The interval modules are in bijection with intervals over $\R$, so 
$V$ is uniquely characterized by a \emph{multiset} called barcode.

Recall that a multiset is a pair $(S,m)$ composed of a set $S$ together with an assignment $m\colon S\rightarrow \Z_{>0} \cup \{\infty\}
$ that maps elements from $S$ to their multiplicity. 
The representation of a multiset $(S,m)$ is the set 
\[\Rep\, (S,m) = \big\{ \; (s,\ell) \;\vert\; s \in S
\mbox{ and } 
\ell\in \llbracket m(s)\rrbracket\; \big\}\,.\]
where $\llbracket n \rrbracket = \{1,2,\ldots, n\}$ for all integers $n>0$.

A persistence module $U$ is said to be \emph{pointwise finite dimensional} (p.f.d) whenever $\dim(U_r)<\infty$ for all $r \in \R$.
In this work, all persistence modules are p.f.d. since they are either of the form $\PH_0(Z)$ for a finite metric space $Z$ or they are $\im(f), \ker(f)$ or $\coker(f)$ for a persistence morphism $f$ between persistence modules (see Subsection~\ref{subsec:persistence-morphism} and Subsection~\ref{subsec:im-ker-coker-intro}). 

Given a persistence module $U$, there is a multiset $\B(U)=(S^{\scst U},m^{\scst U})$, 
where $S^{\scst U}$ is a set of intervals from $\R$ together with an assignment called \emph{multiplicity},
$m^{\scst U}\colon S^{\scst U}\rightarrow \Z_{\geq 0}$, satisfying that there is
an isomorphism 
\[
\mbox{$
U\;
\simeq\; \bigg(\bigoplus_{\scst  (J,\ell) \in \Rep\B(U)}  \kk_{\scst J}\bigg) \oplus \kk_\infty
\;=\;
\bigg(\bigoplus_{\scst J \in S^{\scst U}} 
\bigoplus_{\scst \ell\in \llbracket m^{\scst U}\!(J)\rrbracket}
\kk_{\scst J}\bigg)\oplus \kk_\infty$.}
\]
Besides, if $U$ is the 
0-persistent homology of the Vietoris-Rips filtration of a finite metric space $Z$,  
its barcode, $\B(U)=(S^{\scst U},m^{\scst U})$, is such that all intervals from $S^{\scst U}$ are of the form $[0,b)$ for values $b>0$. 
This is why we might consider $S^{\scst U}$ as a subset of $\R$, where an interval $[0,b)$ is substituted by its endpoint $b$. 
Notice that, in this paper, we do not consider the infinity interval $[0,\infty)$ as an element from $\B(U)$.
Similarly as before, we also write $m^{\scst U}\!(b)$ and $\kk_b$ instead of $m^{\scst U}\!(J)$ and $\kk_{\scst J}$ without ambiguities.
Besides, since $Z$ is finite, we have that $m^{\scst U}\!(b)<\infty$.
Then, 
\[\mbox{$
U\;\simeq\; 
\bigg(\bigoplus_{\scst b>0} 
\bigoplus_{\scst \ell \in \llbracket m^{\scst U}\!(b) \rrbracket}
\kk_b \bigg) \oplus \kk_\infty
$}.
\]


\subsubsection{Persistence morphisms and non-expansive maps}\label{subsec:persistence-morphism}

A \emph{persistence morphism} $f\colon V\rightarrow U$ between persistence modules $V$ and $U$ is a set of linear maps $\big\{f_t\colon V_t\rightarrow U_t\big\}_{t \in \R}$ that commute with the structure maps 
$\rho^{\scst V}$ of $V$ and 
$\rho^{\scst U}$ of $U$. 
That is,
$\rho^{\scst U}_{st}  f_s = f_t \rho^{\scst V}_{st} $ for all $s\leq t$ in $\R$.

Now, consider a pair of finite metric spaces $X$ and $Z$ together with a \emph{non-expansive map} $\fbullet\colon X\rightarrow Z$; that is, there is an inequality 
$d_X(x,y)\geq d_Z(\fbullet(x),\fbullet(y))$ for all $x,y\in X$.
In such case, there is an induced graph morphism $\VR_r(X) \rightarrow \VR_r(Z)$ for all $r\geq 0$. 
In particular, there is an induced persistence morphism $f\colon V\to U$ between   $V=\PH_0(X)$ and $U=\PH_0(Z)$.
Since all intervals of the barcodes $\B(V)$ and $\B(U)$ start at $0$, 
there are no nested intervals.
Hence, we might adapt Theorem~5.3 from~\cite{Jacquard2023} as follows. 

\begin{remark}\label{rem:ulrike}
    Given a    
    persistence morphism $f\colon V\to U$ induced by a non-expansive map $X\to Z$, there exist integers 
    $r^b_a\geq 0$ and $d^+_a,d^-_b\geq 0$ such that
    \[\begin{array}{rr}
    f \simeq  &
    \mbox{$\bigg(
        \bigoplus_{ b > 0} \bigoplus_{a\geq b} 
        \bigoplus_{ i\in \llbracket r^b_a\rrbracket} 
        (\kk_{a} \rightarrow \kk_{b})\bigg) 
        \oplus 
        \bigg(\bigoplus_{b>0} 
        \bigoplus_{j\in\llbracket d^-_b\rrbracket}
        (0 \rightarrow \kk_{b})\bigg)
        $}
    \\
    &
    \mbox{$\oplus 
        \bigg(\bigoplus_{a>0} \bigoplus_{j\in\llbracket d^+_a\rrbracket} 
        (\kk_{a} \rightarrow 0)\bigg) \oplus \big(\kk_{\infty}\rightarrow \kk_{\infty}\big)$.}
    \end{array}
    \]
\end{remark}
Here the persistence morphism $ \kk_{a} \rightarrow \kk_{b}$ satisfies that $ (\kk_{a})_{r} \rightarrow (\kk_{b})_{r}$
is the identity map $\Z_2 \rightarrow \Z_2$ for all $0\leq r < b$ and it is  the null map otherwise.


\subsubsection{Bottleneck distance}

Given a pair of multisets $(S, m)$ and $(S', m')$, a \emph{partial matching} 
\[
\sigma\colon\Rep\, (S,m) \nrightarrow \Rep\, (S', m')
\]
is a bijection $T\rightarrow T'$ for two subsets $T\subseteq \Rep\,(S, m)$ and $T'\subseteq \Rep\,(S', m')$.
For ease, we use the notation $\coim(\sigma)=T$ and $\im(\sigma)=T'$.

Now, consider two barcodes $\B(V)$ and $\B(V')$ for two persistence modules $V$ and $V'$. 
A \emph{$\veps$-matching} between $\B(V)$ and $\B(V')$ is a partial matching $\sigmaV\colon \Rep\B(V)\nrightarrow \Rep\B(V')$ such that:
\begin{itemize}
    \item[1)] all $([a,b), i) \in \Rep \B(V)\setminus \coim(\sigmaV)$ are such that $|b-a|<\veps$,
    \item[2)] all $([a,b), i) \in \coim(\sigmaV)$ are such that $\sigmaV(([a,b), i)) = ([a',b'),j)$ with $|a-a'|<\veps$ and $|b-b'|<\veps$
    \item[3)] all $([a',b'),j) \in \B(V') \setminus \im(\sigmaV)$ are such that $|a'-b'| < \veps$.
\end{itemize}
This leads us the definition of the \emph{bottleneck distance} between $\B(V)$ and $\B(V')$, given by 
\[
\dB(\B(V), \B(V')) = \inf 
\big\{\veps 
 \, |\,
\exists \mbox{ }\veps\mbox{-matching
between } \B(V) \mbox{ and } \B(V')
\big\}.
\]


\subsubsection{Interleaving distance and stability of barcodes
}

    To start, we recall the definition of interleaving distance between two persistence modules.
    Given a persistence module $V$ with structure maps $\rhoV$ and given $\veps\geq 0$, we denote by $\Vveps$ the persistence module such that $\Vveps_r = V_{r+\veps}$ for all $r \in \R$ and with structure maps $\rho^{\scst V(\veps)}_{rs}=\rhoV_{r+\veps, s+\veps}$ for all $r\leq s$.
We denote by $\rhoVveps$ the persistence morphism $\rhoVveps\colon V\rightarrow \Vveps$ given by $(\rhoVveps)_r = \rhoV_{r(r+\veps)}$ for all $r\geq 0$.
Also, given a persistence morphism $f\colon V\rightarrow U$, we denote by $\fveps$ the persistence morphism  $\fveps\colon\Vveps\to \Uveps$ such that $(\fveps)_r = f_{r+\veps}$ for all $r\in \R$.
   
    We say that two persistence modules $V$ and $U$
      are $\veps$-\emph{interleaved} if there exist 
    persistence morphisms, $\phi\colon V\rightarrow \Uveps$ and $\psi\colon U\rightarrow \Vveps$, such that  $\psi_{\veps}\circ \phi = \rho^{\scst V}_{2\veps}$ and also $\phi_{\veps}\circ \psi = \rho^{\scst U}_{2\veps}$.
    The \emph{interleaving distance} between $V$ and $U$ is defined as follows 
    \[
    \dI(V, U) = \inf\big\{ \,\veps\geq 0 \ \big\vert\ V \mbox{ and } U \mbox{ are } \veps\mbox{-interleaved}\,\big\}\,.
    \]
    The following result is adapted from Proposition~7.8 from~\cite{OudotBook2015}.

    \begin{proposition}\label{prop:interleaving-Hausdorff-stability}
    Let $X$ and $X'$ be two finite metric spaces. 
    Let $V=\PH_0(X)$ and $V'=\PH_0(X')$. 
    Then,
    $
    \dI(V, V')\leq 2d_{\scst GH}(X, X')
    $.
    \end{proposition}

   Now we present another result known as the \emph{Isometry Theorem} which is Theorem~3.5
    from~\cite{BauerLesnick2015}.

    \begin{proposition}[\protect{\cite[Theorem~3.5]{BauerLesnick2015}}]\label{prop:isometry}
    Let $V$ and $V'$ be two 
    persistence modules. 
    Then,
    $
    \dB(\B(V), \B(V')) = \dI(V, V')
    $.
    \end{proposition}

    The following result is a consequence of combining Proposition~\ref{prop:interleaving-Hausdorff-stability} with Proposition~\ref{prop:isometry}.

\begin{proposition}\label{prop:bottleneck-Hausdorff-stability}
Let $X$ and $X'$ be two finite metric spaces. Let $V=\PH_0(X)$ and $V'=\PH_0(X')$. 
Then, $\dB(\B(V), \B(V'))\leq 2d_{\scst GH}(X, X')$.
\end{proposition}
   

\subsubsection{Images, kernels and cokernels}\label{subsec:im-ker-coker-intro}

Let $f\colon V\rightarrow U$ be a 
persistence 
morphism.
 We consider the image  persistence module, $\im(f)$,
which is given by $\im(f)_r=\im(f_r)$ for all $r \in \R$ and with structure maps $\rho^{\im(f)}$ being the restriction of $\rho^{\scst U}$. 
That is, $\rho^{\im(f)}_{rs}=\rho^{\scst U}_{rs|\im(f_r)}$ for all $r\leq s$. Similarly, the kernel persistence module, $\ker(f)$, is given by $\ker(f)_r=\ker(f_r)$ for all $r \in \R$ and the structure maps $\rho^{\ker(f)}$ are given by restricting $\rho^{\scst V}$. 
Finally, the cokernel persistence module,
$\coker(f)$, is given by $\coker(f)_r=\coker(f_r)$ for all $r \in \R$ and $\rho^{\coker(f)}$ are induced by the structure maps $\rho^{\scst U}$ in the quotient. 
We might put all these objects in a commutative diagram for $r\leq s$:
\[
\begin{tikzcd}
\ker(f_r)   \ar[r, hookrightarrow]  \ar[d, "\rho^{\ker(f)}_{rs}"]  &
V_r         \ar[r, twoheadrightarrow]  \ar[d, "\rho^{\scst V}_{rs}"]  
\ar[rr, bend left=15, "f_r"] &
\im(f_r)    \ar[r, hookrightarrow]  \ar[d, "\rho^{\im(f)}_{rs}"]  &
U_r         \ar[r, twoheadrightarrow]  \ar[d, "\rho^{\scst U}_{rs}"]  &
\coker(f_r)         \ar[d, "\rho^{\coker(f)}_{rs}"]  
\\
\ker(f_s)   \ar[r, hookrightarrow]  &
V_s         \ar[r, twoheadrightarrow]  
\ar[rr, bend right=15, "f_s", swap] &
\im(f_s)    \ar[r, hookrightarrow]  &
U_s         \ar[r, twoheadrightarrow]  &
\coker(f_s)         
\end{tikzcd}
\]
Image, kernel and cokernel persistence modules are known to be stable under noise perturbations. 
Here, we include a result stating this fact, which, for completeness, we prove in Subsection~\ref{subsec:proof-stability-ker-im-coker}.

\begin{proposition}\label{prop:bottleneck-ker-im-coker}
    Let $Z,Z'$ be a pair of finite metric spaces and consider a pair of subsets $X\subseteq Z$ and $X'\subseteq Z'$ such that $\veps=d_{\scst GH}((X,Z), (X',Z'))
    $. 
    Further, consider $f\colon\PH_0(X)\rightarrow \PH_0(Z)$ and $f'\colon\PH_0(X')\rightarrow \PH_0(Z')$. 
    Then, there are bounds to the bottleneck distances: 
    \begin{align*}
    \dB(\B(\ker(f)), \B(\ker(f'))) & \leq 2\veps, \\
    \dB(\B(\im(f)), \B(\im(f'))) & \leq 2\veps, \\
    \dB(\B(\coker(f)), \B(\coker(f'))) & \leq 2\veps.
    \end{align*}
\end{proposition}


\subsection{The induced block function ${\cM}^0_f$
}~\label{sec:block-function-computation}

The theoretical results supporting the process explained in this section can be consulted in \cite{matchings}. 
There, the authors defined a block function
$
\cM_f\colon S^{\scst V} \times S^{\scst U} \rightarrow \Z_{\geq 0}
$ 
induced by a persistence morphism $f\colon V\to U$ and provided a matrix-reduction algorithm to compute it. 

In our present case, the definition of the block function is simplified as follows. 
Let $f\colon V\to U$ be a persistence morphism 
induced by a non-expansive map $X\rightarrow Z$.
For $I=[0,a)$ and $J=[0,b)$, we write $\cM^0_f(a,b)$ instead of $\cM^0_f([0,a), [0,b))$ to simplify notation. 
We have
\begin{equation}\label{eq:def-Mf}
\cM^0_f(a,b) = 
\dim\bigg(
\dfrac{
f_0(\ker_{a}^+(V)) \cap \ker_{b}^+(U)
}{
f_0(\ker_{a}^-(V)) \cap \ker_{b}^+(U) + f_0(\ker_{a}^+(V)) \cap \ker_{b}^-(U)
}
\bigg)
\end{equation}

This definition has a direct interpretation in terms of barcodes of $\im(f)$, $\ker(f)$ and $\coker(f)$, which we study in Section~\ref{sec:properties-stability}.
In particular, $\cM^0_f(a,b)=0$ for values $a<b$.

\begin{figure}[ht!]
    \centering    \includegraphics[width=0.35\textwidth]{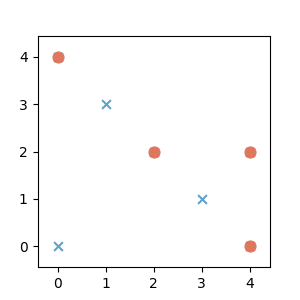}
    \caption{A set $Z_1$ consisting of seven points. The points of $X_1\subset Z_1$ are plotted with red circles while the points from $Z_1\setminus X_1$ are plotted using blue ``x" marks.}   
\label{fig:ex_0_blofun}
\end{figure}

\begin{example}\label{ex:2}
    Consider the pair $X_1\subset Z_1$  depicted in Figure~\ref{fig:ex_0_blofun} that induces a persistence morphism $h \colon \PH_0(X_1)\to \PH_0(Z_1)$. 
    In particular, there are three intervals in $\Rep\B(\PH_0(X_1))$ two with endpoint $a_1=2$ and another with endpoint $a_2 = 2.83$.
   On the other hand, there are six intervals in $\B(\PH_0(Z_1))$, five with endpoint $b_1 =1.41$ and one with endpoint $b_2=2.83$.  
    Using Formula~(\ref{eq:def-Mf}), we obtained that
    $\cM_{h}^0$ 
    and $N_{h}$ are nonzero on the following values
    \[
    \mbox{
    $\cM^0_h(2, 1.41) = 2$, $\cM^0_h(2.83, 1.41)=1$, $N_h(1.41)=2$ and $N_h(2.83)=1$.
    }
    \]
    Here, for convenience, we use approximate values for the endpoints.
\end{example}

An important observation from this paper
is the following.
\begin{remark}
    The definition of $\cM^0_f$ works even in cases when there is no underlying persistence morphism $f$. 
    That is, taking $f_0\colon\Ho_0(\VR_0(X))\to \Ho_0(\VR_0(Z))$ as the linear map induced by the set mapping $X\to Z$ (which does not need to be a non-expansive map), 
    Formula~(\ref{eq:def-Mf}) for $\cM^0_f$ is still well-defined since $f_0(\ker_a^\pm(V))$ and $\ker_b^\pm(U)$ are all subspaces from $\Ho_0(\VR_0(Z))\simeq \langle Z \rangle$.
\end{remark}
Generally, we consider the case when there is an embedding $X\subseteq Z$ and will mention otherwise when we consider $X\rightarrow Z$ as a non-expansive map or just a set mapping.

\begin{figure}[ht!]
    \centering
    \includegraphics[width=0.4\textwidth]{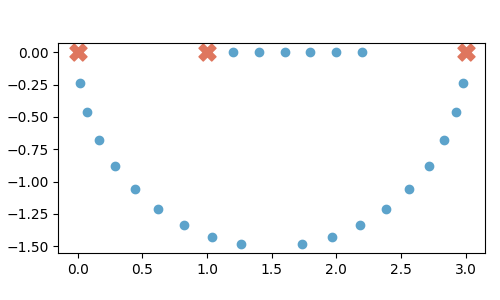}
    \caption{A set $Z_2$ with blue and red colors were the points of $
    g_\bullet(X_2)
    \subset Z_2$ are colored in red. 
    The points of $X_2$ are detailed in Example~\ref{ex:3}. }
    \label{subfig:points_geometric}
\end{figure}

\begin{example}\label{ex:3}
    Consider $X_2=\{(0,0), (0.5, 0), (1,0)\}$ together with a linear map $g_0$ induced by the set mapping $g_\bullet\colon X_2 \rightarrow Z_2$ that sends $X_2$ to the red points depicted in Figure~\ref{subfig:points_geometric}, where $(0,0)\mapsto (0,0)$, $(0.5,0)\mapsto (1,0)$ and $(1,0)\mapsto (3,0)$.
    Using Formula~(\ref{eq:def-Mf}), we obtain that, for values $a_1= 0.5$, $b_3=0.47$ and $b_4=0.8$, we have that 
    $\cM^0_g(a_1,b_3)=\cM^0_g(a_1,b_4)=1$.
    Here, notice that $a_1 = 0.5 < 0.8 =b_4$, contrary to the case when the set mapping induces a persistence morphism.
\end{example}


\subsection{The 0-persistence matching diagram $D(f_0)$}

Now, we briefly present 0-persistence matching diagrams which are a diagram representation for $\cM_f^0$,
analogous to persistence diagrams but with some fundamental differences. 
We have applied 0-persistence matching diagrams in \cite{Df:arxiv} to understand whether a subset ``represents well" the clusters from a larger dataset or not.

Let $\overline{\R}=\R\cup \{\infty\}$.
We define the \emph{0-persistence matching diagram} (also called \emph{matching diagram}, for short) $D(f_0)$ to be the multiset $(\Sf, \mf)$ given by a set $\Sf \subset \overline{\R} \times \R$, consisting of all pairs $(a,b) \in \overline{\R} \times \R$ such that $\cM^0_f(a,b)\neq 0$, 
together with the multiplicity function $\mf$ given by $\mf((a,b))=\cM^0_f(a,b)$ for all pairs $(a,b) \in \R\times \R$ and $\mf((\infty,b))=N_f(b)$ for all $b>0$. 
Notice that, $(a,b) \in \Sf$ if and only if $\cM^0_f(a,b)>0$.
In particular, $(a,b) \in \Sf$ implies $b\leq a$.
Let us see an example of a matching diagram.

\begin{figure}[ht!]
    \centering
    \includegraphics[width=0.4\textwidth]{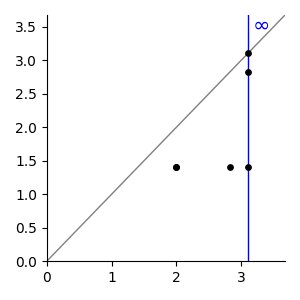}
    \caption{Depiction of the set $\Sf$ associated to the matching diagram $D(f_0)$ considered in Example~\ref{ex:diagram-matching}.
    }
    \label{fig:matching-diagram}
\end{figure}

\begin{example}\label{ex:diagram-matching}
    We consider $\cM^0_f$ 
    from Example~\ref{ex:2}.
    In this case, one can compute the matching diagram $D(f_0)=(\Sf, \mf)$ and plot the points from $\Sf$ as done in Figure~\ref{fig:matching-diagram}. 
    Notice that the pairs $(\infty, b) \in \Sf$ have been plotted over a vertical blue line with the sign $\infty$ next to it. 
    Points $(2, 1.41)$ and $(\infty, 1.41)$ have multiplicity $2$, while other elements from $\Sf$ have multiplicity $1$.
\end{example}


\section{Geometric interpretation of
$\cM^0_f$ 
}\label{sec:geometric-interpretation}

In this section, we give an alternative definition of $\cM^0_f$ that is more geometrically intuitive.
We consider a set mapping $\fbullet \colon X\rightarrow Z$ between finite metric spaces $X$ and $Z$ and the induced linear map $f_0\colon V_0\to U_0$.
We  prove that $\cM^0_f(a,b)$ equals the number of independent cycles in a certain graph that we 
denote by $G(\fbullet)_{(a,b)}$.

First, we use the notation  $\VR^-_s(X)$ and $\VR^+_s(X)$ for the sets  $\cup_{r<s}\VR_r(X)$ and $\cup_{r\leq s}\VR_r(X)=\VR_s(X)$, respectively. 
We consider $\pi_0\colon \textbf{Top}\rightarrow \textbf{Set}$, the functor sending a topological space to the set of its connected components. 

For ease, we fix values $a,b> 0$.
We denote by $\fbullet\big(\VR^+_a(X)\big)$ the graph with vertex set $\fbullet(X)$ and edges $[\fbullet(x), \fbullet(y)]$ for all $[x,y]\in \VR_a^\pm(X)$. 
Next, we consider the graphs 
$
\cA=\fbullet\big(\VR^+_a(X)\big) \cup \VR^-_b(Z)$, 
$
\cB=\fbullet\big(\VR^-_a(X)\big) \cup \VR^+_b(Z)$ and 
$
\cC=\fbullet\big(\VR^-_a(X)\big) \cup \VR^-_b(Z)$,
all of them with vertex set $Z$. Observe that $\cC\subseteq \cA, \cB$.
Besides, we consider the bipartite graph 
$L(\fbullet)$ with vertex set $  \pi_0(\cC)\cup \pi_0(\cA)$ together with the edges $(x,\iota(x))$ for all $x \in \pi_0(\cC)$ and where $\iota\colon\pi_0(\cC)\rightarrow \pi_0(\cA)$ is induced by inclusion. 
In an analogous way, we define the bipartite graph $R(\fbullet)$ with vertex set $\pi_0(\cC)\cup\pi_0(\cB)$ together with the edges $(x,\mu(x))$ for all $x \in \pi_0(\cC) $ and where $\mu\colon\pi_0(\cC)\rightarrow \pi_0(\cB)$
is induced by inclusion. 
Now, observe that the intersection 
$L(\fbullet) \cap R(\fbullet)$ is the vertex set $\pi_0(\cC)$. 
We have the following equivalences.

\begin{remark}\label{re:iso}
    By construction of {$L(\fbullet)$} and {$R(\fbullet)$}, there are isomorphisms
    $\Ho_0(\cA) \simeq \Ho_0(L(\fbullet))$, 
    $\Ho_0(\cB) \simeq \Ho_0(R(\fbullet))$ and
    $\Ho_0(\cC) \simeq \Ho_0(L(\fbullet) \cap R(\fbullet))$.
\end{remark}

Finally, consider the bipartite graph $G(\fbullet)_{(a,b)}$
with vertex set 
$\pi_0({\cal A})\cup\pi_0({\cal C})\cup \pi_0({\cal B})$ containing the edges of $L(\fbullet)$ and $R(\fbullet)$. 
The next remark will be very useful in relating the formula of $\cM_f^0$ to the graphs defined above.

\begin{remark}\label{re:ho-vr}
For all $a,b \in \R$, 
we have that
    \[
\dfrac{
\Ho_0(\VR_0(Z))
}{f_0(\ker_{a}^{\pm}(V)) + \ker_{b}^{\pm}(U)} \simeq 
\Ho_0\Big(\fbullet\big(\VR_{a}^{\pm}(X)\big)\cup \VR_{b}^{\pm}(Z)\Big)\,.
\]
\end{remark}

The following is the main result of this section that we prove after presenting some examples.

\begin{proposition}\label{prop:geometric-matching}
For all $I=[0,a)\in S_V$ and all $J=[0,b) \in S_U$,
\[
\cM^0_f(a,b) = \dim(\Ho_1(G(\fbullet)_{(a,b)}))\,.
\]
where $H_1(G)$ is the group of independent cycles of a given graph $G$.
\end{proposition}

\begin{figure}[ht!]
\begin{center}
\includegraphics[width=0.85\textwidth]{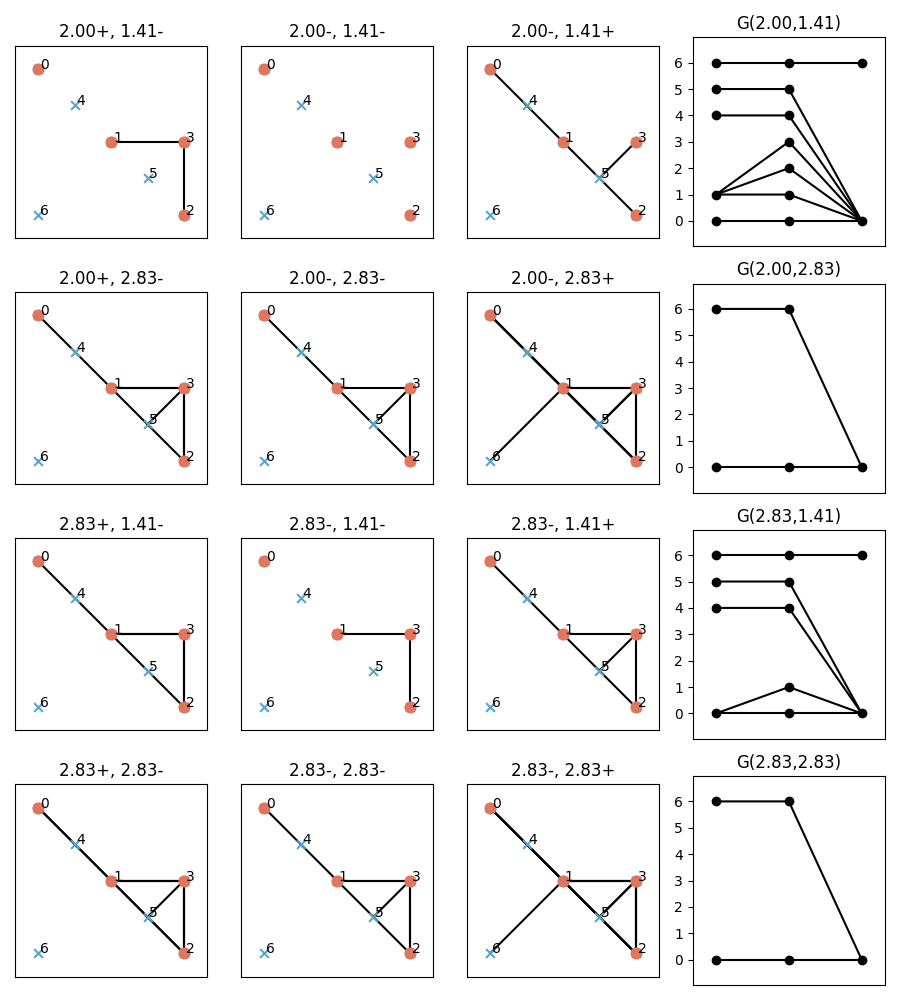}
\end{center}
\caption{
    We depict on each column, from left to right, the graphs ${\cal A}=\VR_a(X_1) \cup \VR^-_b(Z_1)$, ${\cal B}=\VR^-_a(X_1) \cup \VR^-_b(Z_1)$, ${\cal C}=\VR^-_a(X_1) \cup \VR_b(Z_1)$ and $G(h_\bullet)_{(a,b)}$ detailed in Example~\ref{ex:6points-geometric}.
    The graph $G(h_\bullet)_{(a,b)}$ is denoted as $G(a,b)$ for simplicity.
    On the rows, we range over the four pairs, from top to bottom, $(2, 1.41)$, $(2, 2.83)$, $(2.83, 1.41)$ and $(2.83, 2.83)$.
     Notice that, on the first three columns, we have numbered the vertices from $Z_1$. 
    Now, on the rightmost column, we plot $G(h_\bullet)_{(a,b)}$ where a vertex $v=(x,y)$ of $G(h_\bullet)_{(a,b)}$ satisfies that $x=0,1,2$ if $v$ is associated with a connected component $A$ from the set $\pi_0({\cal A}), \pi_0({\cal B}), \pi_0({\cal C})$, respectively, and $y=0,1,2,3,4,5,6$ 
    is given by the smallest index of the points of $A$.
}
\label{fig:geometric-6-points}
\end{figure}

\begin{example}\label{ex:6points-geometric}
    Let $h$ be the persistence morphism from Example~\ref{ex:2} and recall that, for $a_1=2$, $a_2=2.83$ and $b_1=1.41$, we have $\cM^0_h(a_1, b_1)=2$ and $\cM^0_h(a_2,b_1)=1$.
    Now, consider Figure~\ref{fig:geometric-6-points}
    where we can check that $G(h_\bullet)_{(a,b)}$ has no cycle on the pairs, $(a_1,b_2), (a_2,b_2)$, while it has two independent cycles for the pair $(a_1,b_1)$ and a cycle for the pair $(a_2,b_1)$, illustrating that the formula from Proposition~\ref{prop:geometric-matching} holds.
\end{example}

\begin{figure}[ht!]
    \centering
    \includegraphics[width=\textwidth]{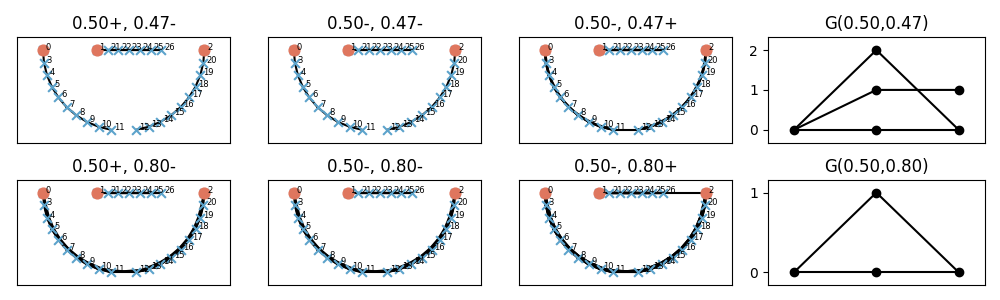}
    \caption{Geometric interpretation of $\cM^0_g$ from Example~\ref{ex:3.2}. 
    The graph $G(g_\bullet)_{(a,b)}$ is denoted as $G(a,b)$ for simplicity.}
    \label{fig:geometric-half-circle}
\end{figure}

 \begin{example}\label{ex:3.2}
    Now, recall Example~\ref{ex:3} where 
    $\cM^0_g(a_1,b_3)=\cM^0_g(a_1,b_4)=1$
    for values $a_1= 0.5$, $b_3=0.47$ and $b_4=0.8$.
    As in Example~\ref{ex:6points-geometric}, we check that Proposition~\ref{prop:geometric-matching} holds by computing the graphs  $G(g_\bullet)_{(a,b)}$, which are plotted in Figure~\ref{fig:geometric-half-circle}, for $(a,b)$ ranging over the pairs $(0.5, 0.47)$ and $(0.5, 0.8)$. 
    As one can see, both $G(g_\bullet)_{(0.5, 0.47)}$ and $G(g_\bullet)_{(0.5,0.8)}$
    contain a single cycle.
\end{example}

Just before proving Proposition~\ref{prop:geometric-matching}, we need to show the following useful result.
\begin{lemma}\label{lemma:alternative-Mf}
Let $A^{\pm}=f_0(\ker^{\pm}_a(V))$ and $B^{\pm}=\ker^{\pm}_b(U)$ where $V=\PH(X)$ and $U=\PH(Z)$.
Then Equation~(\ref{eq:def-Mf}) can be written as
\[\cM_f^0(a,b)=
\dim\bigg(\dfrac{
A^+ + B^-
}{
A^- + B^-
}
\cap 
\dfrac{
A^-+ B^+
}{
A^-+ B^-
}\bigg)
\]
for all $I=[0,a)\in S_V$ and all $J=[0,b) \in S_U$.
\end{lemma}

\begin{proof}
Observe that
$A^-\subset A^+$,
$B^-\subset B^+$ and that  there is a  projection 
 \[
 \pi\colon 
 \dfrac{A^+ \cap B^+}{A^-\cap B^+ +A^+\cap B^- }
 \to \dfrac{ A^- + B^-+A^+\cap B^+ }{ A^- + B^-}
 \]
 that is also an injection because 
 $(A^+\cap B^+)\cap (A^-+B^-)=A^-\cap B^++A^+\cap B^-$
 and then $\ker\pi=0$.
 \\
Moreover, since $(A^+ +B^-)\cap (A^-+B^+)=A^-+B^-+A^+\cap B^+$ then we  have that:
$$\dfrac{
A^- + B^-+A^+\cap B^+
}{
A^- + B^-
}
=
\dfrac{
A^+ + B^-
}{
A^- + B^-
}
\cap 
\dfrac{
A^-+ B^+
}{
A^-+ B^-
}
$$
concluding the proof.
\end{proof}

Now we are ready to prove Proposition~\ref{prop:geometric-matching}.

\begin{proof}[of Proposition~\ref{prop:geometric-matching}]
Let $A^{\pm}=f_0(\ker^{\pm}_a(V))$ and $B^{\pm}=\ker^{\pm}_b(U)$. Recall that in Lemma~\ref{lemma:alternative-Mf} we have proven that 
$$\cM_f^0(a,b)=\dim\bigg(\dfrac{
A^+ + B^-
}{
A^- + B^-
}
\cap 
\dfrac{
A^-+ B^+
}{
A^-+ B^-
}\bigg).$$
Next, we have the isomorphisms 
\[
\dfrac{
A^+ + B^-
}{
A^- + B^-
} 
 \simeq 
 \ker\bigg( 
 \dfrac{
 \Ho_0(\VR_0(Z))
 }{A^- + B^-}
 \rightarrow 
 \dfrac{
 \Ho_0(\VR_0(Z))
 }{A^+ + B^-}
 \bigg)  
 \simeq \ker \big( \Ho_0(\cC) \rightarrow \Ho_0(\cA)\big)
\]
and
\begin{multline*}
\dfrac{
A^- + B^+
}{
A^- + B^-
} 
 \simeq 
 \ker\bigg( 
 \dfrac{
 \Ho_0(\VR_0(Z))
 }{A^- + B^-}
 \rightarrow 
 \dfrac{
 \Ho_0(\VR_0(Z))
 }{A^- + B^+}
 \bigg)  \simeq \ker \big( \Ho_0(\cC) \rightarrow \Ho_0(\cB)\big)\,,
\end{multline*}
where we have used that $ A^{\pm},B^{\pm}\subset \Ho_0(\VR_0(Z))$ and Remark~\ref{re:ho-vr}.
Next, considering the Mayer-Vietoris long exact sequence together with the fact that 
$\Ho_1(L(\fbullet)_{(a,b)})=0$ and 
$\Ho_1(R(\fbullet)_{(a,b)})=0$, we obtain the 
 sequence 
\[\begin{array}{cl}
0 \rightarrow \Ho_1 \big(G(\fbullet)_{(a,b)}\big)&\\
\text{\footnotesize \(\pi_1\)}\!\downarrow&\\
\Ho_0\big(L(\fbullet) \cap R(\fbullet)\big) &
\xrightarrow{\pi_2}
\Ho_0\big(L(\fbullet)\big) \oplus \Ho_0\big(R(\fbullet)\big)\,,
\end{array}
\]
satisfying that $\im (\pi_1)= \ker(\pi_2)$ because it is exact. Using the three isomorphisms given in Remark~\ref{re:iso}, we obtain 
\[
\ker \big( \Ho_0(\cC) \rightarrow \Ho_0(\cA)\big)
\simeq \ker \Big( 
\Ho_0\big(L(\fbullet) \cap R(\fbullet)\big) \rightarrow 
\Ho_0\big(L(\fbullet)\big) \Big)\,,
\]
and also 
\[
\ker \big( \Ho_0(\cC) \rightarrow \Ho_0(\cB)\big)
\simeq \ker \Big( 
\Ho_0\big(L(\fbullet) \cap R(\fbullet)\big) \rightarrow 
\Ho_0\big(R(\fbullet)\big) \Big)\,.
\]
Altogether, we have that
$$
 \Ho_1 \big(G(\fbullet\big)_{(a,b)})\simeq \im (\pi_1)= \ker(\pi_2)\simeq 
 \dfrac{
A^+ + B^-
}{
A^- + B^-
}
\cap 
\dfrac{
A^-+ B^+
}{
A^-+ B^-
}
$$
concluding the proof.
\end{proof}


\section{Properties 
of the induced block function $\cM^0_f$}\label{sec:properties-stability}

In this section, we investigate the distinctive properties of \(\cM_f^0\), which arise from the fact that we work with finite metric spaces and that all the intervals of the barcodes we consider start at zero. 


\subsection{Induced partial matching via $\cM^0_f$ }

In this subsection we prove that a non-expansive map $\fbullet\colon X\to Z$ always induces a unique partial matching  $\sigma^f\colon \Rep \B(V)\nrightarrow\Rep \B(U) $.

We start by highlighting an important property of the induced block 
function $\cM_f^0$. This property can be obtained as a consequence of combining Theorem~4.4 with Proposition~5.3. from~\cite{matchings}, but here, for completeness, we give a direct proof.

\begin{proposition}\label{prop:kappa-a-b-matching}
 Given $b \leq a$,
  $\cM_{\kk_a\to \kk_b}^0(a,b)=1$ and  $\cM_{\kk_a\to \kk_b}^0(a',b')=0$ if either $a'\neq a$ or $b'\neq b$.
\end{proposition}
 \begin{proof}
 First, observe that
 $(\kk_a)_0\to (\kk_b)_0$ is the identity map by definition. Besides,
 \[\mbox{$\ker_{a}^+(\kk_a)=\Z_2=\ker_{b}^+(\kk_b)$ and $\ker_{a}^-(\kk_a)=0=\ker_{b}^-(\kk_b)$.}\]
 Let $f=\kk_a\to \kk_b$, $V=\kk_a$ and $U=\kk_b$.
 Then, 
\[\dfrac{
f_0(\ker_{a}^+(V)) \cap \ker_{b}^+(U)
}{
f_0(\ker_{a}^-(V)) \cap \ker_{b}^+(U) + f_0(\ker_{a}^+(V)) \cap \ker_{b}^-(U)
}=\dfrac{
\Z_2 \cap \Z_2
}{
0\cap \Z_2 + \Z_2 \cap 0
}=\Z_2\,.\]
Now, for 
 $a'< a$, we have that $\ker_{a'}^+(\kk_a)=0$ and then
 $f_0(\ker_{a'}^+(V)) \cap \ker_{b}^+(U)=0$
 so 
 $\cM_{\kk_a\to \kk_b}^0(a',b')=0$.
 On the other hand, if 
 $a< a'$, then 
 $\ker_{a'}^{+}(\kk_a)=\Z_2=\ker_{a'}^{-}(\kk_a)$.
Then, 
$f_0(\ker^-_{a'}(V)) = f_0(\ker^+_{a'}(V))$
and we have 
 \[\dfrac{
f_0(\ker_{a'}^+(V)) \cap \ker_{b'}^+(U)
}{
f_0(\ker_{a'}^-(V)) \cap \ker_{b'}^+(U) + f_0(\ker_{a'}^+(V)) \cap \ker_{b'}^-(U)
}=\dfrac{
f_0(\ker_{a'}^+(V)) \cap \ker_{b'}^+(U)
}{
f_0(\ker_{a'}^+(V)) \cap \ker_{b'}^+(U)
}=0\,,\]
and so $\cM_{\kk_a\to \kk_b}^0(a',b')=0$.
A similar proof can be done for $b'\neq b$. 
\end{proof}

The following result states that $\cM^0_f$ induces a unique partial matching, denoted by $\sigma^f$, whenever $f$ is induced by a non-expansive map. 

\begin{lemma}\label{lemma:unique}
Consider a persistence morphism $f\colon V\to U$ induced by a non-expansive map 
$\fbullet\colon X\rightarrow Z$ between finite metric spaces $X$ and $Z$.
Then, the block function $\cM_f^0$ induces a unique partial matching $\sigmaf\colon \Rep\B(V) \nrightarrow \Rep\B(U)$.
\end{lemma}

\begin{proof}
    To start, by Theorem~4.4. from~\cite{matchings}, we have that $\cM^0_f$ is well-defined. That is,  for any $I=[0,a)\in S^{\scst V}$ with multiplicity $m^{\scst V}\!(a)$, 
    $\cM^0_f$ is such that
\begin{equation}\label{eq:bound1}   
   \mbox{$ \sum_{b\leq a}\cM^0_f(a,b)\leq m^{\scst V}\!(a)$.}
\end{equation}
Besides, Theorem 5.5 from~\cite{matchings} states that if there are no nested intervals then  $\sum\cM_f(I,J)\leq m^{\scst U}(J)$ for any interval $J\in S^{\scst U}$.
Now, since all intervals start at $0$---we may assume that $J=[0,b)$ with multiplicity $m^{\scst U}\!(b)$---we have that
\begin{equation}\label{eq:bound2}
    \mbox{$   \sum_{a\geq b} \cM^0_{f}(a, b) \leq m^{\scst U}\!(b)$.}
\end{equation}
Altogether, by Inequalities~(\ref{eq:bound1}) and~(\ref{eq:bound2}), it follows that $\cM^0_f$ induces a unique partial matching 
$
\sigmaf\colon \Rep\B(V) \nrightarrow \Rep\B(U)
$.
\end{proof}

Later, in~Subection~\ref{subsec:ladder-decomposition}, we see that if the non-expansive map $\fbullet\colon X\rightarrow Z$ is injective then Inequality~(\ref{eq:bound1}) is an equality and $\sigmaf$ is an injection.


\subsection{Ladder 
decompositions 
and the  block function $\cM^0_f$
}~\label{subsec:ladder-decomposition}

In this subsection, we prove that 
a non-expansive injective map $\fbullet \colon X\hookrightarrow Z$ always induces a unique injection $\sigma^f\colon 
\Rep \B(V)\hookrightarrow
\Rep \B(U) $. 
Using this fact, we also relate $\cM^0_f$ to the ladder module structure of $f$ and to the image, kernel and cokernel of $f$.

\begin{proposition}\label{prop:decomposition-f}
Consider a persistence morphism $f\colon V \rightarrow U$ that is induced by a non-expansive injective map $\fbullet\colon X\hookrightarrow Z$. 
Then, we have the following ladder decomposition of $f$,
\begin{equation}
\label{eq:f-direct-sums}
\begin{array}{rl}
f \simeq  &
\mbox{$\bigg(
\bigoplus_{b > 0}
\bigoplus_{a\geq b }
\bigoplus_{s\in\llbracket\cM^0_f(a,b)\rrbracket} (\kk_{a} \rightarrow \kk_{b})\bigg) 
$}\\
&
\mbox{$
\oplus 
\bigg(\bigoplus_{b>0} 
\bigoplus_{s\in \llbracket N_f(b)\rrbracket}
(0 \rightarrow \kk_{b})
\bigg)\oplus
\big( \kk_{\infty}\rightarrow \kk_{\infty}\big)
$}
\end{array}
\end{equation}
where $N_f(b) = m^{\scst U}(b)-\sum_{a>0}\cM^0_f(a,b)$,  and $m^{\scst U}(b)$ denotes the multiplicity of $[0,b)$ in $\B(U)$.
\end{proposition}

\begin{proof}
By Remark~\ref{rem:ulrike}, there are integers $r^b_a\geq 0$ and $d^+_a,d^-_b\geq 0$ such that
\[\begin{array}{rr}
f \simeq  &
\mbox{$
\bigg(
\bigoplus_{b > 0} 
\bigoplus_{a\geq b}
\bigoplus_{s\in \llbracket r^b_a\rrbracket} (\kk_{a} 
\rightarrow \kk_{b})\bigg) 
\oplus 
\bigg(\bigoplus_{b>0} 
\bigoplus_{s\in \llbracket d^-_b\rrbracket}
(0 \rightarrow \kk_{b})\bigg)
$}
\\
&\mbox{$
\oplus 
\bigg(\bigoplus_{a>0} \bigoplus_{s\in \llbracket d^+_b\rrbracket}
(\kk_{a} \rightarrow 0)\bigg) \oplus \big(\kk_{\infty}\rightarrow \kk_{\infty}\big)$}.
\end{array}
\]
Now, by Lemma~\ref{lemma:unique}, the block function $\cM_f^0$ induces a partial matching.
Since $\cM_f^0$ is linear with respect to direct sums and, by Proposition~\ref{prop:kappa-a-b-matching}, $\cM_{\kk_a\to \kk_b}^0(a,b)=1$.
It follows that 
$\cM^0_f(a,b)=r^b_a$ for all $a\geq b>0$. 
On the other hand, it also follows that $d^-_b = N_f(b)$ since there must be as many copies as $m^{\scst U}(b)$ in the codomain of $f$. 
Finally, the induced map 
$
f_0\colon\Ho_0(\VR_0(X)) \rightarrow \Ho_0(\VR_0(Z))
$ 
is an injection since it is equal to the injection $\fbullet\colon X\hookrightarrow Z$. 
This implies that $d^+_a=0$ for all $a>0$, since otherwise $f_0$ would not be injective. 
Altogether we obtain Equation~(\ref{eq:f-direct-sums}).
\end{proof}

 \begin{example}\label{ex:2.1}
    Consider the pair $X_1\subset Z_1$  depicted in Figure~\ref{fig:ex_0_blofun}. From Example~\ref{ex:2}, $\cM_h^0$ 
    and $N_h$ are nonzero on the following values
    \[\mbox{
        $\cM^0_h(2, 1.41) = 2$, 
        $\cM^0_h(2.83, 1.41)=1$, 
        $N_h(1.41)=2$ and 
        $N_h(2.83)=1$.
    }\]
    Then, by Proposition~\ref{prop:decomposition-f}, we obtain the ladder decomposition
\[\begin{array}{rr}
   h  \simeq & 
   \mbox{$(\kk_{2}\rightarrow \kk_{1.41}) \oplus 
   (\kk_{2}\rightarrow \kk_{1.41}) \oplus    (\kk_{2.83}\rightarrow \kk_{1.41}) 
   \oplus
    (0 \rightarrow \kk_{1.41})
    $} 
    \\
    &\mbox{$
    \oplus
    (0 \rightarrow \kk_{1.41})
    \oplus
    (0 \rightarrow \kk_{2.83})
    \oplus \big(\kk_{\infty}\rightarrow \kk_{\infty}\big)
    $.}
    \end{array}
    \]
\end{example}

An immediate consequence of Proposition~\ref{prop:decomposition-f} is that Inequality~(\ref{eq:bound1}) is, in fact, an equality since there are no unmatched intervals by $\cM^0_f$. 
Furthermore, Proposition~\ref{prop:decomposition-f} also implies that $\cM^0_f$ always induces an injective assignment 
$\sigmaf\colon \Rep\B(V) \hookrightarrow \Rep\B(U)$.
This observation leads to the following result.

\begin{proposition}~\label{prop:injective-map}
    Let $\fbullet\colon X\hookrightarrow Z$ be an injective map. Then $\cM^0_f$ induces an injection $\sigma^f\colon \Rep\B(V) \hookrightarrow \Rep\B(U)$.
\end{proposition}

\begin{proof}
To start, given $\delta>0$, we define $X^\delta$ to be the metric space with set $X$ and metric $d_{X\delta}$ given by $d_{X\delta}(x,y)=d_X(x,y)+\delta$ for all $x\neq y$. 
Thus, there is a non-expansive map $X^\delta \rightarrow X$. Further, since we are working with finite metric spaces, taking $\delta>0$ high enough, the composition of $X^\delta\rightarrow X$ with $\fbullet$ is non-expansive; we denote this composition as $\fbullet^\delta$.
Now, by definition of $X^\delta$, it follows that $\ker_a^\pm(V)=\ker_{a+\delta}^\pm(\PH_0(X^\delta))$ for all $a>0$.
Thus, by Equation~(\ref{eq:def-Mf}), we have $\cM_f^0(a,b)= \cM_{f^\delta}^0(a+\delta, b)$ for all $a,b> 0$. 
In particular, we deduce that $\cM_f^0$ induces an injection $ \Rep\B(V) \hookrightarrow \Rep\B(U)$ if and only if $\cM_{f^\delta}^0$ does. 
Now, since $\fbullet^{\delta}$ is injective and non-expansive,
$\cM_{f^\delta}^0$ induces an 
injection  
$ \Rep\B(\PH_0(X^\delta)) \hookrightarrow \Rep\B(U)$
by Proposition~\ref{prop:decomposition-f}, concluding the proof.
\end{proof}

Also, notice that Proposition~\ref{prop:decomposition-f} gives an alternative definition for $\cM^0_f$:
\[
\cM^0_f(a,b) = \# \Big\{ 
\mbox{ summands }
\kk_{a} \rightarrow \kk_{b}
\mbox{ in the ladder decomposition of } f
\Big\}\,.
\]

Besides, $\cM^0_f$ has a direct relation with the interval decomposition for the persistence modules $\ker(f)$, $\coker(f)$ and $\im(f)$, as is shown in the following result.
\begin{proposition}\label{prop:isomorph-im-ker-coker}
There are isomorphisms:
\[\mbox{$
\im(f) \simeq 
   \bigg(\bigoplus_{b>0} \bigoplus_{a\geq b}
    \bigoplus_{s\in\llbracket\cM^0_f(a,b)\rrbracket}
   \kk_{b}\bigg) \oplus \kk_{\infty}$,}
\]
\[\mbox{$\ker(f) \simeq 
    \bigoplus_{b>0} \bigoplus_{a > b} \bigoplus_{s\in\llbracket\cM^0_f(a,b)\rrbracket}
    \kk_{[b,a)}$, $\;$ and also }
\]
\[\mbox{$\coker(f) \simeq \bigoplus_{b>0} 
\bigoplus_{s\in\llbracket N_f(b)\rrbracket}
\kk_{b}$.}\]
\end{proposition}

\begin{proof}
    We prove the  
interval decomposition
    for $\ker(f)$ as the other two are analogous.
    First, notice that, since $a \geq b$, then $\ker(\kk_a\rightarrow \kk_b) = \kk_{[b,a)}$ while $\ker(0\rightarrow \kk_b)=0$ for all $b>0$. 
    Now, by Proposition~\ref{prop:decomposition-f} as well as commutativity of $\ker$ with direct sums, we have the following isomorphisms and equalities
    
    \begin{multline*}
     \ker(f) \simeq 
    \mbox{$\bigg(
    \bigoplus_{b > 0}
    \bigoplus_{a\geq b}  
    \bigoplus_{s\in\llbracket\cM^0_f(a,b)\rrbracket} 
    \ker(\kk_{a} \rightarrow \kk_{b})\bigg) 
    $}
    \\ 
    \qquad \qquad \mbox{$
    = 
    \bigoplus_{b>0} 
    \bigoplus_{a>b}
    \bigoplus_{s\in\llbracket{\cM^0_f(a,b)\rrbracket}} 
    \kk_{[b,a)}
    $}\,.
    \end{multline*}
\end{proof}
In particular, notice that $\cM^0_f$ is determined by the barcodes $\B(\ker(f))$ and $\B(V)$. 
This is proven using the equality $\cM^0_f(a,b) = m^{\ker(f)}([b,a))$ for all $a>b$ and also $\cM^0_f(a,a) = m^{\scst V}(a) - \sum_{b < a} \cM^0_f(a,b)$ for all $a>0$.


\section{Properties and stability of 0-persistence matching diagrams}\label{sec:0-diagrams}

In this section, we study the stability of the matching diagram $D(f_0)$ which, by definition, is equivalent to the stability of $\cM_f^0$. 


\subsection{Relation between $D(f_0)$ and images, kernels and cokernels}

Let us suppose that $f\colon V\to U$ is induced by a 
non-expansive injective map between finite metric spaces $\fbullet\colon X\rightarrow Z$. 
Then, an advantage of using the multiset $D(f_0)$ is that it contains all information from the five multisets $\B(V)$, $\B(U)$, $\B(\ker(f))$, $\B(\im(f))$ and $\B(\coker(f))$. 

Using Propositions~\ref{prop:decomposition-f} and~\ref{prop:isomorph-im-ker-coker} together with the definition of $D(f_0)$ we show the above claim for each case.

\begin{remark}\label{remark:Df)}
Let $f\colon V\to U$
be induced by a 
non-expansive injective map between finite metric spaces.
Let $D(f_0)=(\Sf, \mf)$. 
Then,
\begin{enumerate}
    \item setting $\B(V)=(S^{\scst V}, m^{\scst V})$, we have
    \[
        S^{\scst V}=\{ [0,a) \mbox{ such that }\exists\, (a,b) \in \Sf \mbox{ with } 0< a < \infty\} 
    \]
    and, given $0<a<\infty$, 
    $
        m^{\scst V}(a) = \sum_{b\leq a} \mf((a,b))
    $;     
    \item  similarly, given $\B(U)=(S^{\scst U}, m^{\scst U})$ we have 
    \[
        S^{\scst U}=\{ [0,b) \mbox{ such that } \exists\, (a,b) \in \Sf \mbox{ with } b<\infty \mbox{ and } b \leq a \leq \infty\}
    \]
    and also 
    $
        m^{\scst U}(b) = \sum_{b\leq a\leq \infty} \mf((a,b))
    $;
    \item  given $\B(\ker(f))=(S^{\ker(f)}, m^{\ker(f)})$ we have 
    \[
        S^{\ker(f)}=\{ [b,a) \mbox{ such that } \exists\, (a,b) \in \Sf \mbox{ with } b < a < \infty\}
    \]
    and also $m^{\ker(f)}([b,a)) = \mf((a,b))$;
    \item  given $\B(\im(f))=(S^{\im(f)}, m^{\im(f)})$ we have 
    \[
        S^{\im(f)}=\{ [0,b) \mbox{ such that } \exists \,(a,b) \in \Sf \mbox{ with } b \leq a < \infty\}
    \]
    and also 
    $
        m^{\im(f)}(b) = \sum_{b\leq a < \infty} \mf((a,b))
    $ for $b< \infty$;
    \item given $\B(\coker(f))=(S^{\coker(f)}, m^{\coker(f)})$ we have 
    \[
        S^{\coker(f)}=\{ [0,b) \mbox{ such that } \exists \,(\infty,b) \in \Sf\}
    \]
    and also $m^{\coker(f)}([0,b)) = \mf((\infty,b))$.
\end{enumerate}
\end{remark}

Conversely, we might write $ D(f_0)$ in terms of $\B(\ker(f))$, $\B(V)$ and $\B(\coker(f))$. This is followed by setting 
\begin{multline*}
\mbox{$\Sf = \big\{ 
(a,b) \mbox{ such that } [b,a) \in S^{\ker(f)} \big\} $}
 \\
\hspace{2cm} \mbox{$\bigcup 
 \big\{ (a,a) \mbox{ such that } a<\infty \mbox{ and }  m^{\scst V}(a) > \sum_{b<a}m^{\ker(f)}([b,a))\big\} $}
 \\
 \mbox{$\bigcup \big\{(\infty, b) \mbox{ such that } [0,b) \in S^{\coker(f)}
 \big\}$}
\end{multline*}
and also 
\[
\mf((a,b)) = \begin{cases}
    m^{\ker(f)}([b,a)) & 
    \mbox{ if } b < a < \infty\,, 
    \\
    m^{\scst V}(a) - \sum_{r<a}m^{\ker(f)}([r,a)) & 
    \mbox{ if } b=a < \infty \,,
    \\
    m^{\coker(f)}(b) & 
    \mbox{ if } b<a=\infty\,.
\end{cases}
\]

\subsection{
Stability of $D(f_0)$ induced by an embedding $X\subseteq Z$
}\label{subsec:stability}

In this subsection, we find an upper bound for the stability of $\cM^0_f$.
We do this for the case of an embedding $X\subseteq Z$.

Intuitively, by the stability of $\cM^0_f$ we mean that considering the pair of samples $X \subseteq Z$ and moving their positions slightly, so that we obtain a new pair $X'\subseteq Z'$, then the induced block functions $\cM^0_f$ and $\cM^0_{f'}$ are approximately ``the same'', where $f'\colon V'\to U'$ for 
$V'=\PH_0(X')$ and $U'= \PH_0(Z')$.

In particular, we show the stability of $D(f_0)$, which, by its definition, is equivalent to showing the stability of $\cM^0_f$, but it is easier to formulate.

\begin{theorem}\label{th-stability}
    Consider finite metric spaces $X\subseteq Z$ and $X'\subseteq Z'$ such that $d_{\scst GH}((X,Z),(X',Z')) = \veps$. 
    Then there exists a partial matching $\sigma^{\scst D(f_0)}\colon \Rep D(f_0)\nrightarrow \Rep D(f_0')$ such that 
    \begin{itemize}
        \item for $((a,b), i) \in \coim(\sigma^{\scst D(f_0)})$, writing $\sigma^{\scst D(f_0)}((a,b), i) = ((a',b'),j)$, we have that (either $a=a'=\infty$ or $|a-a'| \leq 2\veps$) and $|b-b'| \leq 2\veps$.
        \item for $((a,b), i) \in \Rep D(f_0) \setminus \coim(\sigma^{\scst D(f_0)})$ then (either $a=\infty$ or $a \leq 2\veps$) and $b \leq 2\veps$.
        \item for $((a',b'), j) \in \Rep D(f_0') \setminus \im(\sigma^{\scst D(f_0)})$ then (either $a'=\infty$ or $a' \leq 2\veps$) and $b' \leq 2\veps$.
    \end{itemize}
   We call such a partial matching a $2\veps$-matching. 
\end{theorem}

Observe that there is no confusion with $\veps$-matchings between barcodes, since here we refer to partial matchings between matching diagrams.
The proof of this theorem, which is quite technical, is elaborated in detail in 
Subsection~\ref{subsec:th-proof}.

\begin{figure}[ht!]
    \centering
    \begin{subfigure}{0.4\textwidth}
        \includegraphics[width=\textwidth]{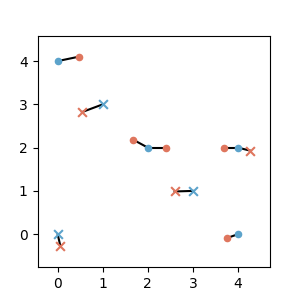}
    \caption{$X\subset Z$ (in red) and $X'\subset Z'$ (in blue).}
    \label{subfig:stability-points}
    \end{subfigure}
    \begin{subfigure}{0.5\textwidth}
    \includegraphics[width=\textwidth]{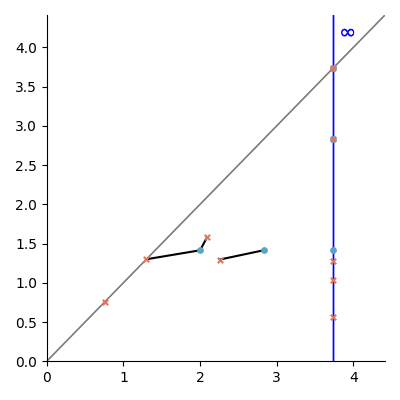}
    \caption{
    $D(f_0)$ (in red) and $D(f_0')$  (in blue).
    }
    \label{subfig:stability-bimodules}
    \end{subfigure}
    \caption{
    Illustration of the stability of matching diagrams.
    On the left, we show two pairs $(X,Z)$ and $(X',Z')$ and the  matching that provides the value of $d_{\scst GH}((X,Z),(X',Z'))$. 
    The matching diagrams $D(f_0)$ and $D(f_0')$ associated to $X\subset Z$  and $X'\subset Z'$, respectively, are depicted on the right with the partial matching $\sigma^{\scst D(f_0)}$.
    }
\end{figure}

\begin{example}\label{ex:stability}
    Consider the pairs $X \subseteq Z$ 
    and $X'\subseteq Z'$ such that $d_{\scst H}(Z, Z') < \veps$ and $d_{\scst H}(X, X') < \veps$ for $\veps \sim 0.5$ depicted in 
    Subfigure~\ref{subfig:stability-points}.
    We plot $Z$ in blue and $Z'$ in red, and we plot the points of $X$ and $X'$ using circles while the remaining points are marked with $\times$ signs. In the same subfigure, we pair the points from $Z$ with the closest ones from $Z'$ by connecting them with lines; this matches points from $X$ with points from $X'$. 
    Now, in Subfigure~\ref{subfig:stability-bimodules}, we plot, in a superposed manner, the matching diagrams $D(f_0)$ and $D(f_0')$, respectively. 
    We can observe that there is a $2\veps$-matching $\sigma^{\scst D(f_0)}\colon
    D(f_0)\nrightarrow D(f_0')$.
    More concretely, we have $D(f_0)=\{\Sf, \mf\}$ and $D(f_0')=\{\Sff, \mff\}$ where the pairs $(a,b)\in \Sf$ and $(a',b')\in \Sff$ range over the values in the table below (where values are rounded to two decimals):
        \[
    \begin{array}{rr|rr|rr}
        a\;\; &    b\;\; &   a'\;\;  &   b'\;\;  &  \;\;\vert a-a'\vert &   \;\;\vert b-b'\vert \\ \hline
    2.00 & 1.41  & 2.09 & 1.58 & 0.09    &    0.17 \\
    2.00 & 1.41 & 1.30 & 1.30  & 0.7    &    0.11 \\
    -\;\; & -\;\; & {\bf 0.75} & {\bf 0.75} & -\;\; & -\;\; \\
   \infty\;\;
    & 2.83 & \infty\;\;
    & 2.83 & \textrm{nan}     &    0.00 
    \\
   \infty \;\;
    & 1.41 & \infty\;\;
    & 1.28 & \textrm{nan}      &    0.13 \\
    \infty\;\;
    & 1.41 & \infty\;\;
    & 1.03 & \textrm{nan}      &    0.38 \\
    -\;\;   & -\;\; & \infty\;\;
    & 0.57 & -\;\;      &    -\;\; 
    \end{array}
    \]
    Further, in this table, in each row, a couple of pairs of values $(a,b)$ and $(a',b')$ are matched, and the differences between their endpoints are shown on the right hand of the table. 
    Notice that the maximum value is $0.75$, so the stability bound is satisfied since $0.75 < 2\cdot 0.5=1$.
    In Subfigure~\ref{subfig:stability-bimodules}, we depict such matching by connecting segments.
\end{example}

Finally, notice that our stability result implies that the ``noise" on $D(f_0)$ lies around the points $(0,0)$ and $(\infty, 0)$.

\subsection{Stability of
$D(f_0)$ induced by set injections $X\hookrightarrow Z$ 
}\label{sec:set-injections-matchings}

We conclude this section by adapting Theorem~\ref{th-stability} to the case of a set injection of finite metric spaces $X\hookrightarrow Z$ that might not be an embedding.
To do this, we extend the definition of the Gromov-Hausdorff distance to pairs of set injections of finite metric spaces.

\begin{example}\label{ex:digits-PCA}
    Consider samples from a handwritten digits dataset~\cite{misc_optical_recognition_of_handwritten_digits_80}. This dataset contains about 380 samples of images for each of the digits from $0$ to $9$. Let $Z$ be the set consisting of the samples labelled by either $0$ or $9$; we plot their PCA projection on two components in Figure~\ref{fig:pca_mnist_09_dim2}.
    On the other hand, we consider the subset $X$ containing those samples labelled by $0$. 
    We denote by $Z(n)$ to be the PCA projection of $Z$ into its $n$ principal components. 
    In an analogous manner, we use the notation $X(n)$. 
    Here, observe that, even though there is a natural set inclusion $X(n) \hookrightarrow Z(n)$ for all $n$, such inclusion might alter the metrics arbitrarily. 
    In Figure~\ref{fig:pca_mnist_09_dimN} we consider the matching between $X(n)$ and $Z(n)$ for various values of $n>0$. In particular, notice that some points are above the diagonal. 
    We observe that smaller values of $n$ lead to more spread points, while, for larger values, the points group around the diagonal. 
    One can interpret this as saying that, the larger the $n$, the more similarities will be between the clusters from $X(n)$ and their counterparts in $Z(n)$.
\end{example}

\begin{figure}[h!]
    \centering
    \includegraphics[width=0.5\linewidth]{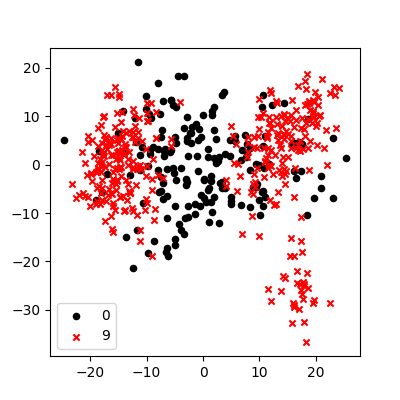}
    \caption{Projection of two principal components of $Z$, where samples are labelled by $0$ or $9$.}
    \label{fig:pca_mnist_09_dim2}
\end{figure}

\begin{figure}[h!]
    \centering
    \includegraphics[width=0.95\linewidth]{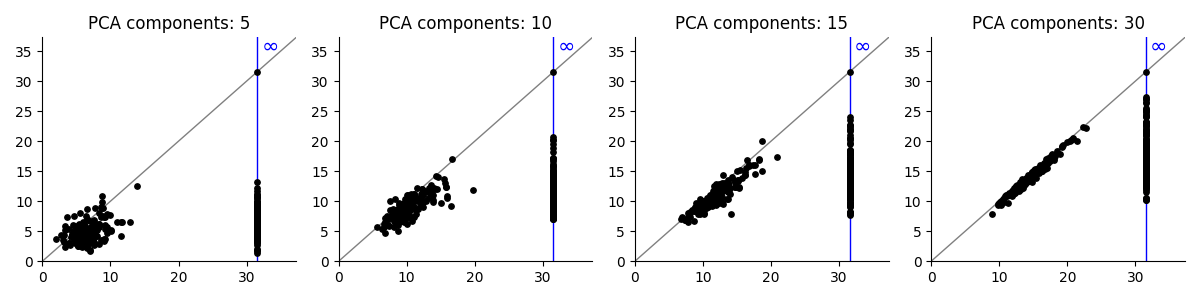}
    \caption{Matching diagram between $X(n)$ and $Z(n)$ for varying $n \in \{ 5, 10, 15, 30\}$.}
    \label{fig:pca_mnist_09_dimN}
\end{figure}

Next, we want to adapt Theorem~\ref{th-stability} to the setup from this section. 
Unfortunately, Definition~\ref{def:gromov-hausdorff-pairs} of Gromov-Hausdorff distance only works for the case of two embeddings $\fbullet\colon X\subseteq Z$ and $\fbullet'\colon X'\subseteq Z'$.
Instead, if $\fbullet\colon X
\hookrightarrow Z$ or $\fbullet'\colon X'\hookrightarrow Z'$ is not an embedding, 
we might formulate stability in terms of assignments when the condition $\#X=\#X'$ is fulfilled.

\begin{definition}\label{def:veps-interleaving-sets}
    Consider metric spaces $X,X',Z,Z'$ together with set injections
    $\fbullet \colon X\hookrightarrow Z$ and $\fbullet'\colon X'\hookrightarrow Z'$. 
    Given $\veps>0$, we say that $\fbullet$ and $\fbullet'$ are $\veps$-interleaved if the following conditions hold:
    \begin{enumerate}
        \item there exist an isomorphism $\alpha^X\colon X\rightarrow X'$
    inducing a $\veps$-interleaving between $V=\PH_0(X)$ and $V'=\PH_0(X')$. That is, there are inclusions
    \[
    \mbox{
    $\alpha^X(\VR_r(X))\subseteq \VR_{r+\veps}(X')$, 
    $(\alpha^X)^{-1}(\VR_r(X'))\subseteq \VR_{r+\veps}(X)$
    }
    \]
    which induce maps $\phi^X\colon V\rightarrow V'(\veps)$ and $\phi^{X'}\colon V'\rightarrow V(\veps)$ that are such that $\phi^{X'}_\veps\circ \phi^X = \rho^V_{2\veps}$ and $\phi^{X}_\veps\circ \phi^{X'} = \rho^{V'}_{2\veps}$.
        \item There exist assignments $\alpha^Z\colon Z\rightarrow Z'$ and  $\alpha^{Z'}\colon Z'\rightarrow Z$ inducing a $\veps$-interleaving between $U=\PH_0(Z)$ and $U'=\PH_0(Z')$ via maps $\psi^Z\colon U\rightarrow U'(\veps)$ and $\psi^{Z'}\colon U'\rightarrow U(\veps)$,
        \item $\alpha^Z \circ \fbullet = \fbullet' \circ \alpha^X$ and $\alpha^{Z'} \circ \fbullet' = \fbullet \circ (\alpha^X)^{-1}$
    \end{enumerate}
\end{definition}

Notice that Definition~\ref{def:veps-interleaving-sets} does not generalize nor is a particular case of 
Definition~\ref{def:gromov-hausdorff-pairs}.
Next, we explain how assignments come up naturally, following Example~\ref{ex:digits-PCA}. 

\begin{example}
    We let $X$ and $Z$ to be the respective metric spaces $X(n)$ and $Z(n)$ (for some $n>0$) as defined in Example~\ref{ex:digits-PCA}. Now, consider metric spaces $X'=(X, d^{X'})$ and $Z'=(Z, d^{Z'})$ whose metrics are equal to $d^X$ and $d^Z$ respectively, up to a perturbation by a parameter $\veps>0$. 
    That is, given $x_1,x_2 \in X$, we have $|d^X(x_1,x_2)-d^{X'}(x_1,x_2)|\leq\veps$ and, given $z_1, z_2 \in Z$, we have $|d^Z(z_1,z_2)-d^{Z'}(z_1,z_2)|\leq\veps$. 
    In this case, $\alpha^X, (\alpha^X)^{-1}, \alpha^{Z}$ and $\alpha^{Z'}$ are given by the underlying set identities. 
    These assignments satisfy the conditions of Definition~\ref{def:veps-interleaving-sets} and so there is a $\veps$-interleaving between $\fbullet\colon X\rightarrow Z$ and $\fbullet'\colon X'\rightarrow Z'$.
\end{example}

Now we are ready to present the main theorem from this section.

\begin{theorem}\label{theorem:stability-general}
    Consider metric spaces $X,X',Z,Z'$ together with set injections 
    $\fbullet\colon X\hookrightarrow Z$ and $\fbullet'\colon X'\hookrightarrow Z'$ that are $\veps$-interleaved.
    Then, there is a $\veps$-matching between $D(f_0)$ and $D(f'_0)$.
\end{theorem}

The proof of this theorem is given in Subsection~\ref{subsec:5.2} for being rather technical.


    \section{
    Proof of Proposition~\ref{prop:bottleneck-ker-im-coker},
    and Theorems~\ref{th-stability} and~\ref{theorem:stability-general}}\label{sec:proofs}

We devote a dedicated section for the proofs of Proposition~\ref{prop:bottleneck-ker-im-coker},
and Theorems~\ref{th-stability} and~\ref{theorem:stability-general}
for being rather technical.


    \subsection{
    Proof of Proposition~\ref{prop:bottleneck-ker-im-coker}}
    \label{subsec:proof-stability-ker-im-coker}

    The aim of this subsection is to show the stability of $\im(f)$, $\ker(f)$ and $\coker(f)$. 
    The original result of this fact can be found in~\cite{kernels-images-cokernels}.
    However, the Stability Theorem shown in~\cite{kernels-images-cokernels} is formulated in terms of pairs of continuous functions $g,g'\colon L\rightarrow \R$ given an underlying topological space $L$ and such that $||g-g'||\leq\veps$. 
    Thus, we would need to adapt the existing Stability Theorem to the Gromov-Hausdorff distance and the case of persistence morphisms induced by an inclusion.
    Instead, we choose to prove such stability by a simple argument, which aids in making the article self-contained. 
      
    To start, we define the interleaving distance between persistence morphisms. 
    Let $f\colon V\rightarrow U$ and $f'\colon V'\rightarrow U'$ be a pair of persistence 
    morphisms.
    We say that $f$ and $f'$ are $\veps$-interleaved if there exists four persistence 
    morphisms, 
    $\phi\colon V\rightarrow V'(\veps)$, $\phi'\colon V'\rightarrow \Vveps$, $\psi\colon U\rightarrow U'(\veps)$ and $\psi'\colon U'\rightarrow \Uveps$ such that:
     \begin{itemize}
         \item 
       $\phi'_{\veps} \circ \phi=\rho^{\scst V}_{2\veps}$, $\phi_{\veps} \circ \phi'=\rho^{\scst V'}_{2\veps}$,
    \item $\psi'_{\veps} \circ \psi=\rho^{\scst U}_{2\veps}$,
    $\psi_{\veps} \circ \psi'=\rho^{\scst U'}_{2\veps}$,
   \item  $
   \psi \circ f=f'_{\veps} \circ \phi$ and $\psi' \circ f'=\fveps \circ 
   \phi'$.
    \end{itemize}
     Then the interleaving distance between $f$ and $f'$ is:
    \[
    \dI(f,f')=\inf \big\{\, \veps\geq 0 \,\mid\, f \mbox{ and } f' \mbox{ are $\veps$-interleaved\,}\big\}\,.
    \]
    Now, assuming that $f$ and $f'$ are $\veps$-interleaved, we have the following result. 
    \begin{proposition}\label{prop:interleaving-im-ker-coker}
        Given $\veps = \dI(f,f')$, then 
        \begin{align*}
            \dI(\im(f), \im(f')) & \leq \veps\,, \\
            \dI(\ker(f), \ker(f')) & \leq \veps\,, \mbox{ and } \\
            \dI(\coker(f), \coker(f')) & \leq \veps\,. 
        \end{align*}
    \end{proposition}
    
    \begin{proof}
    Let us see the first inequality. 
   The commutativity relation $\psi\circ f = f'_{\veps}\circ \phi$ implies 
   $\psi(\im(f)) \subseteq \im(f')(\veps)$. 
   Also, $\psi'\circ f' = \fveps\circ \phi'$ 
   implies 
   $\psi'(\im(f')) \subseteq \im(f)(\veps)$. 
   Hence, $\psi$ and $\psi'$ restrict to an $\veps$-interleaving between $\im(f)$ and $\im(f')$.
\\
    Next, we consider the following commutative diagram
   with exact rows
    \[
    \begin{tikzcd}
        \im(f) 
        \ar[r, hookrightarrow] 
        \ar[d, "\psi_{|\im(f)}"] 
        & 
        U 
        \ar[r, twoheadrightarrow] 
        \ar[d, "\psi"] 
        & 
        \coker(f) 
        \ar[d, "\overline{\psi}"] 
        \\
        \im(f')(\veps) 
        \ar[r, hookrightarrow]  
        & 
        U'(\veps) 
        \ar[r, twoheadrightarrow] 
        & 
        \coker(f')(\veps)
    \end{tikzcd}
    \]
    where the arrow $\overline{\psi}$ is such that it commutes in the diagram. 
    By a similar commutativity argument, one can define 
    $    \overline{\psi'}\colon \coker(f') \rightarrow \coker(f)(\veps)
    $. 
    Also, since 
    $
    \psi'_{\veps}\circ \psi = \rho^{\scst U}_{2\veps}
    $, 
    it follows that 
    $    \overline{\psi'_{\veps}} \circ \overline{\psi}     
    = \rho^{\coker(f)}_{2\veps}
    $. 
    Similarly, we also have 
    $    \overline{\psi^{\,\veps}} \circ \overline{\psi'}
    = \rho^{\coker(f')}_{2\veps}
    $.
    Hence, $\coker(f)$ and $\coker(f')$ are $\veps$-interleaved.
    \\
    Finally, by equalities 
    $
    \psi\circ f = f'_{\veps}\circ \phi
    $ 
    and 
    $
    \psi'\circ f' = \fveps\circ \phi'
    $ 
    we have inclusions 
    $
    \phi(\ker(f)) \subseteq \ker(f')(\veps)
    $ and 
    $    \phi'(\ker(f'))\subseteq \ker(f)(\veps)$. 
    Altogether, $\phi$ and $\phi'$ restrict to an $\veps$-interleaving between $\ker(f)$ and $\ker(f')$.
    \end{proof}

Next, we adapt Proposition~\ref{prop:interleaving-Hausdorff-stability} to persistence morphisms induced by inclusions.   
In particular, the following result works for 0-persistent homology.
    \begin{proposition}\label{prop:interleaving-f-fprime}
    Let $X\subseteq Z$ and $X'\subseteq Z'$ be finite metric spaces. 
    Let $V=\PH_0(X)$, 
    $V'=\PH_0(X')$,
    $U=\PH_0(Z)$
    and
    $U'=\PH_0(Z')$.
    Let $f\colon V\to U$ and $f'\colon V'\to U'$ be the persistence morphisms induced by the inclusions. 
    Then, 
    $
    \dI(f, f') \leq 2 d_{\scst GH}((X,Z), (X',Z'))
    $.
    \end{proposition}
    
    \begin{proof}
        Let $\veps = d_{\scst GH}((X,Z), (X',Z'))$. 
        There exists a metric space $M$ together with a pair of isometries $\gamma_{\scst Z}\colon Z\hookrightarrow M$ and $\gamma_{\scst Z'}\colon Z'\hookrightarrow M$ such that $\dHM(\gamma_{\scst Z}(Z), \gamma_{\scst Z'}(Z')) \leq \veps$ and $\dHM(\gamma_{\scst Z}(X), \gamma_{\scst Z'}(X')) \leq \veps$.
        Thus, we consider an assignment 
        $
        \alpha^{\scst X}\colon X\rightarrow X'
        $ 
        sending a point $x\in X$ to 
        $
        \alpha^{\scst X}(x)\in X'
        $ 
        such that $
        \dM(\gamma_{\scst Z}(x),\gamma_{\scst Z'}(\alpha^{\scst X}(x)
        ))\leq \veps
        $. 
        Then, we extend $\alpha^{\scst X}$ to an assignment $\alpha^{\scst Z}\colon Z\rightarrow Z'$ such that 
        $
        \alpha^{\scst Z}_{\vert \scst X} = \alpha^{\scst X}
        $ 
        and that 
        $
        \dM(\gamma_{\scst Z}(z),\gamma_{\scst Z'}(\alpha^{\scst Z}(z)
        ))\leq \veps
        $ for all $z \in Z$. 
        \\
        Similarly, we define an assignment $\alpha^{{\scst Z}'}\colon Z'\rightarrow Z$ such that 
        $\dM(\gamma_{\scst Z}(\alpha^{{\scst Z}'}(z'))$, $ \gamma_{\scst Z'}(z'))\leq \veps$ for all $z' \in Z'$ and  $\alpha^{{\scst Z}'}(X') \subseteq X$, so we define $\alpha^{{\scst X}'}=\alpha^{{\scst Z}'}_{\vert \scst X'}$. 
        \\
        Such assignments determine well-defined morphisms of graphs,   
        $
        \alpha^{\scst Z}\colon \VR_r(Z) $ $\rightarrow \VR_{r+2\veps}(Z')$ and $\alpha^{{\scst Z}'}\colon \VR_r(Z') \rightarrow \VR_{r+2\veps}(Z)$,
        for all $r\geq 0$,
        where, for example, given an edge $[u,v] \in \VR_r(Z)$ we have $
        \alpha^{\scst Z}([u,v]) = [
        \alpha^{\scst Z}(u), \alpha^{\scst Z}(v)] \in \VR_{r+2\veps}(Z')$; here notice that we have used that $\gamma_{\scst Z}$ is an isometry together with the triangle inequality  on $\dM$:
        \begin{multline*}
             d^{\scst Z'}(\alpha^{\scst Z}(u), \alpha^{\scst Z}(v)) = 
        \dM(\gamma_{\scst Z'}(\alpha^{\scst Z}(u)), \gamma_{\scst Z'}(\alpha^{\scst Z}(v))) \\
        \leq \dM(\gammaZp(\alpha^{\scst Z}(u)), \gamma_{\scst Z}(u)) + \dM(\gamma_{\scst Z}(u),\gamma_{\scst Z}(v)) + \dM(\gamma_{\scst Z}(v),\gamma_{\scst Z'}(\alpha^{\scst Z}(v))) \\ \leq \veps + r + \veps = r+2\veps\,.
         \end{multline*}
         Further, notice that $\alpha^{\scst X}$ and $\alpha^{{\scst X}'}$
        are such that 
        $
        \alpha^{\scst X}(\VR_r(X)) \subseteq \VR_{r+2\veps}(X')
        $ and 
        $
        \alpha^{{\scst X}'}(\VR_r(X')) \subseteq \VR_{r+2\veps}(X)$.
\\     
        Altogether, by construction,
        $\alpha^{\scst X}$ and $\alpha^{\scst Z}$        
        induce 
        persistence 
        morphisms,
        $\phi^{\scst X}\colon V\rightarrow V'(2\veps)$ and 
        $\psi^{\scst Z}\colon U\rightarrow U'(2\veps)$, such that 
        $\psi^{\scst Z} \circ f=f_{2\veps}'\circ \phi^{\scst X}$.
        Similarly, $\alpha^{{\scst X}'}$ and $\alpha^{{\scst Z}'}$ induce persistence morphisms,
         $
         \phi^{\scst X'}\colon V'\rightarrow V(2\veps)
         $ and $
         \psi^{\scst Z'}\colon U'\rightarrow U(2\veps)$, such that $\psi^{\scst Z'} \circ f'=f_{2\veps} \circ \phi^{\scst X'}$. 
         \\
         Also, by construction of $\alpha^{{\scst Z}}$ and $\alpha^{{\scst Z}'}$, we have
         the other equalities that make that $f$ and $f'$ are $2\veps$-interleaved, concluding that 
         $\dI(f,f')\leq 2\veps$.
    \end{proof}

Now we are ready to prove Proposition~\ref{prop:bottleneck-ker-im-coker}.
    
\begin{proof}[of Proposition~\ref{prop:bottleneck-ker-im-coker}]
 First, since $\veps = d_{\scst GH}((X,Z), (X',Z'))$ then, by Proposition~\ref{prop:interleaving-f-fprime}, we have that $\dI(f,f')\leq 2\veps$. 
 Now, we obtain the inequalities in the statement of the proposition directly by applying Proposition~\ref{prop:interleaving-im-ker-coker}.
\end{proof}


\subsection{Proof of Theorem~\ref{th-stability}}\label{subsec:th-proof}

As the subsection indicates, here our aim is to prove Theorem~\ref{th-stability}. 
To start, we consider two sub-multisets from $D(f_0)=(\Sf, \mf)$.
We denote by $D(f_0)_{\scst O}$ the sub-multiset of $D(f_0)$ given by the pairs $(a,b) \in \Sf$ with $a<\infty$ and with multiplicity $\mf((a,b))$. 
On the other hand, we denote by $D(f_0)_\infty$ the multiset of pairs $(\infty,b) \in \Sf$ with multiplicity $\mf((\infty,b))$. 
In particular, one has the disjoint union $D(f_0)=D(f_0)_{\scst O} \sqcup D(f_0)_\infty$.  
We prove Theorem~\ref{th-stability} based on the following two results.

\begin{proposition}\label{prop:rep-Dfinfty}
    Given $\veps = d_{\scst GH}((X,Z),(X',Z'))$, there exists a $2\veps$-matching
    \[
    \sigmainf\colon \Rep D(f_0)_\infty \nrightarrow \Rep D(f_0')_\infty\,.
    \]    
\end{proposition}

\begin{proof}
    To start, we recall that there are a couple of bijections $\mu\colon\Rep D(f_0)_\infty $ $\rightarrow \Rep \B(\coker(f))$  given by sending $((\infty,b),i) \in \Rep D(f_0)_\infty$ to $([0,b),j) \in \Rep \B(\coker(f))$ and $\mu'\colon\Rep D(f_0')_\infty \rightarrow \Rep \B(\coker(f'))$ which is analogous. 
    Now, by Propostion~\ref{prop:isometry} and Propostion~~\ref{prop:interleaving-im-ker-coker}, we have $\dB(\coker(f),\coker(f')) \leq 2\veps$. 
    Consequently, there exists a partial matching $\sigma^{\coker}\colon \Rep \B(\coker(f)) \nrightarrow \Rep \B(\coker(f'))$ that sends $([0,b),i)$ to $([0,b'), j)$ with $|b-b'| \leq 2\veps$ or, given $([0,b),i) \in \Rep \B(\coker(f))$ unmatched, then $b \leq 2\veps$ and given $([0,b'),i) \in \Rep \B(\coker(f'))$ unmatched then $b' \leq 2\veps$.
\end{proof}

The second result that we need requires more work, and we delay the proof for later. 

\begin{proposition}\label{prop:rep-DfO}
    Given $\veps = d_{\scst GH}((X,Z),(X',Z'))$, there exists a $2\veps$-matching
    \[
    \sigmaO\colon\Rep D(f_0)_{\scst O} \nrightarrow \Rep D(f_0')_{\scst O}\,.
    \]    
\end{proposition}

Now, we are ready to prove Theorem~\ref{th-stability}.

\begin{proof}[of Theorem~\ref{th-stability}]
    We consider a partial matching $\sigma^{\scst D(f_0)}\colon D(f_0)\nrightarrow D(f_0')$ that is the combination of $\sigmainf$ and $\sigmaO$ produced in Propositions~\ref{prop:rep-Dfinfty} and~\ref{prop:rep-DfO}. 
    Since $D(f_0)$ is the disjoint union $D(f_0) = D(f_0)_{\scst O}\sqcup D(f_0)_\infty$ then, by construction of $\sigmainf$ and $\sigmaO$, the  partial matching
    $\sigma^{\scst D(f_0)}$ is a $2\veps$-matching,    
    concluding the proof.
\end{proof}

The rest of this section is devoted to prove Proposition~\ref{prop:rep-DfO}. 
For this, we consider fixed finite metric spaces $Z$ and $Z'$ and subsets $X\subseteq Z$ and $X'\subseteq Z'$ such that $\veps = d_{\scst GH}((X,Z),(X',Z'))$ and persistence morphisms $f\colon V\to U$ and $f'\colon V'\to U'$ induced by the inclusions.
Recall the assignments $\alpha^{\scst Z}$ and $\alpha^{\scst Z'}$ from Proposition~\ref{prop:interleaving-im-ker-coker}. 
We pay close attention to their restrictions 
$
\alpha^{\scst X}=\alpha^{\scst Z}_{|\scst X}
\colon X\rightarrow X'
$ and $
\alpha^{\scst X'}=\alpha^{\scst Z'}_{|\scst X'}\colon X'\rightarrow X
$ and use them to define a subset $Y\subseteq X\times X'$ given by pairs $(x,x')\in X \times X'$ such that either $\alpha^{\scst X}(x)=x'$ or $\alpha^{\scst X'}(x')=x$. 
Fixed $\delta>0$, we denote by $Y^\delta$ the metric space with set $Y$ and metric $d^{\scst Y\delta}$ where, given $(x,x'), (y,y') \in Y^\delta
$ we have 
\[
d^{\scst Y\delta}((x,x'), (y,y')) = \begin{cases}
    0 & \mbox{ if } (x,x')=(y,y') \\
    \delta  + d^{\scst X}(x,y) & \mbox{ otherwise } 
\end{cases}
\]
We can also define $\Ydp$ with distance $d^{\scst Y\delta'}$ where, given $(x,x'), (y,y') \in Y
$, we have 
\[
d^{\scst Y\delta\, '}((x,x'), (y,y')) = \begin{cases}
    0 & \mbox{ if } (x,x')=(y,y') \\
    \delta  + d^{\scst X'}(x',y') & \mbox{ otherwise } 
\end{cases}
\]
There are non-expansive maps $
\pi\colon Y^\delta \rightarrow X$ and $\pi'\colon\Ydp\rightarrow X'$ given by 
$\pi((x,x'))=x$ and $\pi'((x,x'))=x'$ for all $(x,x') \in Y$. 
Now, we define the persistence modules $W= \PH_0(Y^\delta)$ and $W'= \PH_0(\Ydp)$ and consider the persistence morphisms $h\colon W\rightarrow V$ and $h'\colon W'\rightarrow V'$ induced by $\pi$ and $\pi'$, respectively. 
We write $f^\delta=f\circ h\colon W\to U$ and $\fdp=f'\circ h'\colon W'\to U'$.

\begin{proposition}~\label{prop:interleaving-f-delta}
    $\dI(f^\delta, \fdp) \leq 2\veps$ for any $\delta > 0$.
\end{proposition}

\begin{proof}
    Let us see that $f^{\delta}$ and $f^{\delta'}$ are $2\veps$-interleaved.
    First, we can define two morphisms of graphs, for all $r\geq 0$,  
    $
   \piYd\colon\VR_r(Y^\delta) \rightarrow \VR_{r+2\veps}(\Ydp)
   $ and $
   \piYdp\colon\VR_r(\Ydp) \rightarrow \VR_{r+2\veps}(Y^\delta)
   $ 
   by using  the identity on $Y$.
   To show that $\piYd$ is well-defined, consider different pairs $(x,x'), (y,y') \in Y$ such that $d^{X}(x,y)=r$, there is an inequality 
   
   \begin{multline*}
   \dYdp((x,x'),(y,y')) - \delta 
   = d^{\scst X'}(x',y')
   = \dM(\gammaZp(x'), \gammaZp(y'))
   \\ \;\;\;\;\;\leq 
   \dM(\gammaZp(x'), \gammaZ(x)) + \dM(\gammaZ(x), \gammaZ(y)) + \dM(\gammaZ(y), \gammaZp(y'))\leq 
   r + 2\veps\,. 
   \end{multline*}
   
   Similarly, $\piYdp$ is well-defined.
   In turn, these induce a pair of persistence morphisms $\phi^{\scst Y}\colon W \rightarrow W'(2\veps)$ and $
   \phi^{\scst Y'}\colon W' \rightarrow W(2\veps)$ such that 
   $$
   \phi^{\scst Y'}_{2\veps}\circ \phi^{\scst Y}= \rho^{\scst W}_{4\veps}
   \;\mbox{ and }\;
   \phi^{\scst Y}_{2\veps}\circ \phi^{\scst Y'} = \rho^{\scst W'}_{4\veps}
   .$$ 
    On the other hand, there is a  commutative relation 
   $
   \phi^{\scst X}\circ h=h'_{2\veps}\circ \phi^{\scst Y}
   $ which we need to prove. 
   First, notice that,
   for any $(x,x') \in Y$ we have that
   $[\alpha^{\scst X}(x), x']$ is an edge 
     in $\VR_{2\veps}(X')$ (or even $\alpha^{\scst X}(x) =x'$), and, in particular, given $r\geq 0$, the class $[(x,x')] \in W_r$ is such that there is an equality in $V'_{r+2\veps}$:
   \[
   \phi^{\scst X}_{r+2\veps} \circ h_r([(x,x')]) =
   [\alpha^{\scst X}(x)] = 
   [x'] = 
   (h'_{2\veps})_r\circ \phi^{\scst Y}_r([(x,x')])\,.
   \]
   As $\big\{ [(x,x')] \ | \ (x,x') \in Y\big\}$ is a set of generators for $W_r$, we  conclude $(\phi^{\scst X} \circ h)_r = (h'_{2\veps}\circ \phi^{\scst Y})_r$. 
   By the same principle,  
   $\phi^{\scst X'}\circ h'=h_{2\veps}\circ \phi^{\scst Y'}$. 
   Using these relations, and since, by Proposition~\ref{prop:interleaving-f-fprime}, we have 
   $
   f'_{2\veps} \circ \phi^{\scst X} = \psi^{\scst Z} \circ f$ and $f_{2\veps} \circ \phi^{\scst X'} = \psi^{\scst Z'} \circ f'
   $ we can see that: 
   \[
   \fdp_{2\veps} \circ \phi^{\scst Y} = 
   f'_{2\veps} \circ h'_{2\veps} \circ \phi^{\scst Y} = 
   f'_{2\veps} \circ \phi^{\scst X} \circ h =
   \psi^Z \circ f \circ h 
   =
   \psi^{\scst Z} \circ f^\delta
   \]
   as well as 
   \[
   f^\delta_{2\veps} \circ \phi^{\scst Y'} = 
   f_{2\veps} \circ h_{2\veps} \circ \phi^{\scst Y'} = 
   f_{2\veps} \circ \phi^{\scst X'} \circ h' = 
   \psi^{Z'} \circ f' \circ h'
   =
   \psi^{\scst Z'} \circ f^{\delta'}\,.
   \]
   This completes the proof since we know by 
   Proposition~\ref{prop:interleaving-f-fprime}
   that 
   $\psi^{\scst Z'}_{2\veps} \circ \psi^{\scst Z} = \rho^{\scst U}_{4\veps}$ and also $\psi^{\scst Z}_{2\veps} \circ \psi^{\scst Z'} = \rho^{\scst U'}_{4\veps}$.
\end{proof}

The following result shows that  all intervals from $\B(\ker(f^\delta))$ have length greater or equal to $\delta$.

\begin{lemma}\label{lem:Bkerfdelta}
All intervals from $\B(\ker(f^\delta))$ are of the form $[b,a+\delta)$ with $0<b\leq a$ or  $[0,\delta)$. 
Besides, an interval $[b,a+\delta)\in \B(\ker(f^\delta))$ with $0<b\leq a$ has multiplicity 
$m^{f\delta}([b,a+\delta)) = \cM_f^0(a,b)$  and 
$m^{f\delta}([0,\delta))=\dim(\ker(h_{0}))$. 
\end{lemma}
\begin{proof}
First, by definition of $W$, it follows that $\ker^{\pm}_b(W) = 0$ for $b < \delta$, while 
$\ker^+_{\delta}(W) = \ker(h_0)$.
In addition, since $f^\delta_0 = f_0 \circ h_0$ and $\ker(f_0)=0$, it follows that $\ker(f^\delta_0)=\ker(h_0)$. Using this observation, together with Remark~\ref{rem:ulrike} and commutativity of $\cM^0_{f^\delta}$ with direct sums, there is an isomorphism
\[\begin{array}{rr}
f^\delta \simeq  &
\mbox{$
\bigg(
\bigoplus_{b > 0} 
\bigoplus_{a\geq b}
\bigoplus_{s\in \llbracket \cM^0_{f^\delta}(a,b)\rrbracket} (\kk_{a} 
\rightarrow \kk_{b})\bigg) 
\oplus 
\bigg(\bigoplus_{b>0} 
\bigoplus_{s\in \llbracket d^-_b\rrbracket}
(0 \rightarrow \kk_{b})\bigg)
$}
\\
&\mbox{$
\oplus 
\bigg(
\bigoplus_{s\in \llbracket d^+_\delta\rrbracket}
(\kk_{\delta} \rightarrow 0)\bigg) \oplus \big(\kk_{\infty}\rightarrow \kk_{\infty}\big)$}.
\end{array}
\]
for  $d^+_\delta = \dim(\ker(h_0))$ and some integers $d^-_b\geq 0$ for all $b>0$. 

Now, we observe that by definition of $W$:  
\[
h_0(\ker^{\pm}_b(W)) = \begin{cases}
    0 & \mbox{ for } b \leq \delta\,, \\
    \ker^{\pm}_{b-\delta}(V) & \mbox{ for } \delta < b\,.
\end{cases}
\]
After these observations, using the partial matching formula we have 
\[
\cM^0_{f^\delta}(a, b) = \begin{cases}
    0 & \mbox{ for } a \leq \delta\,, \\
    \cM_f^0(a-\delta,b) & \mbox{ for } \delta < a\,.
\end{cases}
\]
In particular, this also implies that $N_{f^\delta}(b)=N_f(b)$ for all $b>0$.
Now, using our decomposition for $f^\delta$, there is an isomorphism 
\[
\ker(f^\delta) \simeq 
\bigg( \bigoplus_{b > 0} 
\bigoplus_{a\geq b}
\bigoplus_{s\in \llbracket \cM^0_f(a,b)\rrbracket} 
\kk_{[b, a+\delta)}
\bigg) \oplus 
\bigg(
\bigoplus_{s\in \llbracket d^+_\delta\rrbracket}
\kk_{\delta}
\bigg)
\] 
and so $\B(\ker(f^\delta))$ is the multiset $(S^{f\delta}, m^{f\delta})$ given by 
\[
S^{f\delta}
= \{ [b,a+\delta) \mbox{ such that } \cM_f^0(a,b)\neq 0\} \cup \{ [0,\delta)\} 
\] 
and with multiplicities 
$m^{f\delta}([b,a+\delta)) = \cM_f^0(a,b)$ for $b>0$ and 
$m^{f\delta}([0,\delta))=\dim(\ker(h_{0}))$. 
\end{proof}

The following result is why we are interested in $f^\delta$ and is the last ingredient we need to prove Proposition~\ref{prop:rep-DfO}.

\begin{proposition}\label{prop:injection-DfO-ker}
There exists an injection 
\[
\sigma\colon 
\Rep D(f_0)_{\scst O} \hookrightarrow \Rep \B(\ker(f^\delta)) 
\]
such that $((a,b), i) \in D(f_0)_{\scst O}$ is sent to $([b,a+\delta),j) \in \Rep \B(\ker(f^\delta))$ and such that given $([c,d), l) \in \B(\ker(f^\delta)) \setminus \im \sigma$ then $c=0$ and $d=\delta$. 
\\
There is also an analogous  injection $\sigma'\colon\Rep D(f_0')_{\scst O} \hookrightarrow \Rep \B(\ker(\fdp))$.
\end{proposition}

\begin{proof}
The result follows from Lemma~\ref{lem:Bkerfdelta}, recalling that $D(f_0)_{\scst O}$ is the multiset given by pairs $(a,b)\in \R\times \R$ with multiplicities $m^{\scst D(f_0)}=\cM_f(a,b)$.
\end{proof}

Finally, we are ready to provide the proof of Proposition~\ref{prop:rep-DfO}.

\begin{proof}[of Proposition~\ref{prop:rep-DfO}]
First, we fix $\delta > 2\veps$. 
By Proposition~\ref{prop:interleaving-f-delta},  $\dI(f^\delta,\fdp)\leq 2\veps$.
In particular, by Proposition~\ref{prop:interleaving-im-ker-coker} we also have that
$\dI(\ker(f^\delta)$, $\ker(\fdp))\leq 2\veps$ 
which in turn, by Proposition~\ref{prop:isometry}, implies that $\dB(\B(\ker(f^\delta))$, $\B(\ker(\fdp)) \leq 2\veps$. 
Thus, there exists a $2\veps$-matching $\sigma^{\ker}\colon\Rep \B(\ker(f^\delta))$ $\nrightarrow \B(\ker(\fdp))$. 
Besides, since all elements $([a,b), i) \in \Rep \B(\ker(f^\delta))$ are such that $|a-b|\geq \delta > 2\veps$ and all $([a',b'),j) \in \Rep \B(\ker(\fdp))$ are such that $|a'-b'|\geq \delta > 2\veps$, it follows that $\sigma^{\ker}$ must be a bijection sending $([a,b),i) \in \Rep \B(\ker(f^\delta))$ to $([a',b'),j) \in \Rep \B(\ker(\fdp))$ with $|a-a'| \leq 2\veps$ and $|b-b'| \leq 2\veps$. 
Now, the claimed partial matching $\sigmaO$ follows by the commutative square, where $\sigma,\sigma'$ are from applying Proposition~\ref{prop:injection-DfO-ker} to $f^\delta$ and $\fdp$
\[
\begin{tikzcd}
    \Rep D(f_0)_{\scst O} \ar[r, "\sigma", hookrightarrow] \ar[d, "\sigmaO", near start] \ar[d, phantom, "/" marking] & 
    \Rep \B(\ker(f^\delta)) \ar[d, "\sigma^{\ker}", "\simeq"'] 
    \\
    \Rep D(f_0')_{\scst O} \ar[r, "\sigma'", hookrightarrow] & \Rep \B(\ker(\fdp))\,.
\end{tikzcd}
\]
That is, $\sigmaO$ is such that:
\begin{itemize}
\item If $((a,b), i) \in \Rep D(f_0)_{\scst O}\setminus \coim \sigmaO$, then $\sigmaO((a,b), i) = ([b,a+\delta), j)$ and $\sigma^{\ker}([b,a+\delta), i') = ([0,\delta), k)$, which implies that $a = |(a+\delta) - \delta| \leq 2\veps$.
\item For $((a,b), i) \in \coim \sigmaO$, then $\sigmaO((a,b), i) = ((a',b'),j) \in \Rep D(f_0')_{\scst O}$ such that $\sigma ((a,b), i) = ([b, a+\delta), i')$ and $\sigma' ((a',b'), j) = ([b', a'+\delta), j')$ with $\sigma^{\ker} ([b, a+\delta), i') = ([b', a'+\delta), j')$. This implies $|b-b'|\leq 2\veps$ and $|a-a'| = |(a+\delta)-(a'+\delta)| \leq 2\veps$.
\item If $((a',b'), j) \in \Rep D(f_0')_{\scst O}\setminus \im \sigmaO$, then $\sigma((a',b'), j) = ([b',a'+\delta), j')$ and $\sigma^{ker}(([0,\delta), k))= ([b',a'+\delta), j')$, which implies that $a' = |(a'+\delta) - \delta| \leq 2\veps$.
\end{itemize}
Therefore, the partial matching 
constructed fulfills all the requirements needed in the statement of Proposition~\ref{prop:rep-DfO}, concluding the proof.
\end{proof}


\subsection{Proof of Theorem~\ref{theorem:stability-general}}\label{subsec:5.2}

    Essentially, the proof is analogous to that outlined in Subsection~\ref{subsec:th-proof}. 
    The proof is divided into three parts: 
    (1) we define two persistence morphisms $f^\delta$ and $f^{\delta'}$, 
    (2) we show that $d_I(f^\delta, f^{\delta'}) \leq \veps$ and 
    (3)  we construct a $\veps$-matching between $D(f_0)$ and $D(f_0')$. Let us proceed.
    \begin{enumerate}
        \item Using $\alpha^X$ and $\alpha^{X'}=(\alpha^{X})^{-1}$,
        we consider $Y^\delta$ and $Y^{\delta'}$ from Subsection~\ref{subsec:th-proof}. 
        In this case, $Y^\delta=(Y, d^{Y\delta})$ is composed of
        \[
            Y= \big\{(x,\alpha^X(x)) \ \big\vert \ x \in X\big\}
        \]
        together with the distance $d^{Y\delta}$ such that for $x_1,x_2\in X$ with $x_1\neq x_2$ we have $d^{Y\delta}((x_1, \alpha^X(x_1)), (x_2, \alpha^X(x_2)))=d^X(x_1, x_2) + \delta$.
        Now, for any $\delta>0$, there are non-expansive maps $\pi\colon Y^\delta \rightarrow X$ and $\pi'\colon Y^{\delta'} \rightarrow X'$. 
        As we are working with finite metric spaces, we take $\delta>0$ large enough such that $f_0 \circ \pi\colon Y^\delta \rightarrow Z$ and $f_0' \circ \pi'\colon Y^{\delta'} \rightarrow Z'$ are non-expansive. 
        In such case, these compositions induce persistence morphisms $f^\delta\colon W\rightarrow U$ and  $f^{\delta'}\colon W'\rightarrow U'$ where $W=\PH_0(Y^\delta)$ and $W'=\PH_0(Y^{\delta'})$.
        
        \item Next, we show that $d_I(f^\delta, f^{\delta'})\leq \veps$. 
        As in Proposition~\ref{prop:interleaving-f-delta}, we consider morphisms of graphs for all $r\geq 0$,  
        $
        \piYd\colon\VR_r(Y^\delta) \rightarrow \VR_{r+\veps}(\Ydp)
        $ and $
        \piYdp\colon\VR_r(\Ydp) \rightarrow \VR_{r+\veps}(Y^\delta)
        $ 
        by using the identity on $Y$.
        To show that $\piYd$ is well-defined, consider different pairs $(x,\alpha^X(x)), (y,\alpha^X(y)) \in Y^\delta$ such that $d^{X}(x,y) = r$. 
        Now,  using the inclusion $\alpha^X(\VR_r(X))\subseteq \VR_{r+\veps}(X')$, there is an inequality,
        \[
            \dYdp((x, \alpha^X(x)), (y, \alpha^X(y)) - \delta 
            = d^{\scst X'}(\alpha^X(x), \alpha^X(y))
            \leq d^{\scst X}(x, y) + \veps\,, 
        \]
        which implies $\alpha^{Y\delta}(\VR_{r+\delta}(Y^\delta))\subseteq \VR_{r+\delta+\veps}(Y^{\delta'})$.
        Similarly, $\piYdp$ is well-defined.
        In turn, these induce a pair of persistence morphisms $\phi^{\scst Y}\colon W \rightarrow W'(\veps)$ and 
        $
        \phi^{\scst Y'}\colon W' \rightarrow W(\veps)$ such that 
        $
        \phi^{\scst Y'}_{\veps}\circ \phi^{\scst Y}= \rho^{\scst W}_{2\veps}
        $ and $
        \phi^{\scst Y}_{\veps}\circ \phi^{\scst Y'} = \rho^{\scst W'}_{2\veps}$. 
        In addition, since $\alpha^{Y\delta}$ and $\alpha^{Y\delta'}$ are defined as the identity on $Y$, we deduce commutativity relations $\alpha^X \circ \pi = \pi' \circ \alpha^Y$ and $\pi \circ \alpha^{Y'} = \alpha^{X'} \circ \pi'$.

        By hypotheses, we already have a pair of assignments $\alpha^Z\colon Z\rightarrow Z'$ and $\alpha^{Z'}\colon Z'\rightarrow Z$ inducing a $\veps$-interleaving between $U$ and $U'$ via maps $\psi^Z\colon U\rightarrow U'(\veps)$ and $\psi^{Z'}\colon U'\rightarrow U(\veps)$. 
        Finally, by the commutativity $\alpha^Z \circ f_0 = f_0' \circ \alpha^X$ together with $\alpha^X \circ \pi = \pi' \circ \alpha^Y$, we obtain 
        \[
        \alpha^Z \circ (f_0 \circ \pi) = f_0' \circ \alpha^X \circ \pi = (f_0' \circ \pi') \circ \alpha^{Y}\,.
        \]
       which implies $\psi^Z \circ f^\delta=f^{\delta'} \circ \phi^Y$.
        Similarly, using $\alpha^{Z'} \circ f_0' = f_0 \circ \alpha^{X'}$, we obtain that $\psi^{Z'}\circ f^{\delta'} = f^{\delta} \circ \phi^{Y'}$.

        \item Now, since $\fbullet\colon X\rightarrow Z$ is an injection and $\pi\colon Y^\delta\rightarrow X$ is an isomorphism, we obtain 
        $\ker(f^\delta_0)= \ker(f_0 \circ \pi)=0$. 
        Thus, 
        since $f_0 \circ \pi\colon Y^\delta \rightarrow X$ is a non-expansive injection, using Proposition~\ref{prop:decomposition-f}, we obtain
        \[
        \begin{array}{rr}
        f^\delta \simeq  &
        \mbox{$
        \bigg(
        \bigoplus_{b > 0} 
        \bigoplus_{a\geq b}
        \bigoplus_{s\in \llbracket \cM^0_{f^\delta}(a,b)\rrbracket} (\kk_{a} 
        \rightarrow \kk_{b})\bigg) 
        $ \hspace{3cm}}
        \\
        &\mbox{$
        \oplus 
        \bigg(\bigoplus_{b>0} 
        \bigoplus_{s\in \llbracket N_{f^\delta}(b)\rrbracket}
        (0 \rightarrow \kk_{b})\bigg)
        \oplus \big(\kk_{\infty}\rightarrow \kk_{\infty}\big)$}\,.
        \end{array}
        \]
        Next, using that $\cM^0_f(a,b)=\cM^0_{f^\delta}(a+\delta,b)$ for all $a,b>0$, we have 
        \[
        \ker(f^\delta) \simeq 
        \bigoplus_{b > 0} 
        \bigoplus_{a > 0}
        \bigoplus_{s\in \llbracket \cM^0_{f}(a,b)\rrbracket} \kk_{[b, a+\delta)}\,.
        \]
        In addition, since $N_{f^\delta}(b) = N_f(b)$ for all $b>0$, there is an isomorphism 
        $
        \coker(f^\delta) \simeq 
        \bigoplus_{b>0} \bigoplus_{s\in \llbracket N_f(b)\rrbracket}\kk_{b}
        $.
        
        Now, we consider $D(f_0)=D(f_0)_O \sqcup D(f_0)_\infty$, where $D(f_0)_O$ consists of $(a,b)\in D(f_0)$ with both $a,b<\infty$ while $D(f_0)_\infty$ consists of points $(\infty, b)\in D(f_0)$. 
        In particular, notice that there is an isomorphism $D(f_0)_O \simeq \B(\ker(f^\delta))$ that is built sending $(a,b) \in D(f_0)_O$ to $[b,a+\delta)$ while there is another isomorphism $D(f_0)_\infty \simeq \B(\coker(f^\delta))$ defined sending $(\infty, b) \in D(f_0)_\infty$ to $[0,b)$ in $\B(\coker(f^\delta))$.
        Altogether, since $d_I(f^\delta,f^{\delta'})\leq\veps$, using Proposition~\ref{prop:interleaving-im-ker-coker} we obtain a pair of $\veps$-matchings $\B(\ker(f^\delta)) \nrightarrow \B(\ker(f^{\delta'})$ and 
        $\B(\coker(f^\delta))\nrightarrow \B(\coker(f^{\delta'}))$.
        Thus, we have obtained a $\veps$-matching $D(f_0)_\infty\nrightarrow D(f_0')_\infty$. 
        On the other hand, we take $\delta>0$ large enough so that 
        $a+\delta - b>\veps$ for all $a,b>0$ such that $\cM_f(a,b)\neq 0$.
        This implies that the $\veps$-matching $\B(\ker(f^\delta)) \nrightarrow \B(\ker(f^{\delta'})$ cannot pair points into the diagonal (from outside) and, in particular, it is a matching. 
        Thus, we obtain a $\veps$-matching $D(f_0)_O\rightarrow D(f_0')_O$, and the result follows.
    \end{enumerate}


\section{Conclusion and future works}
\label{sec:future}

In this paper, we 
provide theoretical results that confirm that the definition of 0-persistence matching diagram  $D(f_0)$ is well-founded.
In particular, we have demonstrated the stability of $D(f_0)$ for $f$ induced by an embedding of finite metric spaces, through algebraic properties.

There has been recent work on Gromov-Hausdorff distances between pairs~\cite{GomezChe2024}; a future research direction is to see whether our results hold for their alternative definition of Gromov-Hausdorff distance (and check if our alternative definition is equivalent to theirs).
In our future work, we will explore more scenarios where $X$ is not a subset of $Z$.
We also plan to relax the condition from Definition~\ref{def:veps-interleaving-sets}. 
Morevoier, as pointed out in Proposition~\ref{prop:isomorph-im-ker-coker}, $\cM_f^0$ has a close relation with $\ker(f)$, $\im(f)$ and $\coker(f)$. 
Now, in the examples studied in Section~\ref{sec:set-injections-matchings} there is no persistence morphism, and hence there are no images, kernels or cokernels. 
We plan to answer the following questions:
Can we consider such $\cM_f^0$ as a generalization of images, kernels or cokernels? What about cases when the persistence modules are more complex, i.e., they contain nested intervals?

Finally, in the future, we will also investigate topological features of dimensions higher than $0$ where we can find nested intervals and not all intervals start at $0$, as well as other type of filtrations. 

\section*{Acknowledgments}

This work received funding from the Horizon Europe
research and innovation program under grant agreement no. 101070028-2
\href{https://rexasi-pro.spindoxlabs.com/}{REXASI-PRO} ``REliable \& eXplainable Swarm Intelligence for People with Reduced mObility" (call HORIZON-CL4-2021-HUMAN-01-01), and from 
MCIN/AEI/10.13039/ 501100011033 and the European Union NextGenerationEU/PRTR, under the 
national project TED2021-129438B-I00 ``Computational Topology for energy saving and optimization of deep learning methods to achieve Green Artificial Intelligence solutions".

\section{Code avaliablity}

The source code for the examples and experiments described in this paper is available on the GitHub repository \url{https://github.com/Cimagroup/tdqual}
and the explanation of the algorithm developed can be found in \cite{Df:arxiv}.

\section*{ORCID}

\noindent Álvaro Torras-Casas - \url{https://orcid.org/0000-0002-5099-6294}

\noindent Rocío González-Díaz - \url{https://orcid.org/0000-0001-9937-0033}

\bibliographystyle{plain}

\bibliography{library}

\end{document}